\documentclass[10pt,twocolumn,twoside]{IEEEtran}

\usepackage{color}
\definecolor{mygreen}{RGB}{10, 128, 128}

\usepackage{hyperref}
\hypersetup{
	colorlinks   = true, %Colours links instead of ugly boxes
	urlcolor     = black, %Colour for external hyperlinks
	linkcolor    = mygreen, %Colour of internal links
	citecolor   = mygreen %Colour of citations
}

\usepackage{amsmath,amsthm,amssymb,mathtools,nomencl}
\usepackage{algorithm,algorithmic}

\usepackage{booktabs,threeparttable}
\usepackage{multirow}
 \usepackage{graphicx,epstopdf}
 \DeclareGraphicsExtensions{.ps}
\usepackage{color}
\definecolor{myblue}{RGB}{0, 0, 0}

\usepackage{array}
\newcolumntype{P}[1]{>{\centering\arraybackslash}p{#1}}
\newcolumntype{M}[1]{>{\centering\arraybackslash}m{#1}}

\allowdisplaybreaks

\usepackage{tikz}
\usetikzlibrary{shapes,arrows}
\usepackage[tight,TABTOPCAP]{subfigure}

\usepackage{graphicx,subfigure,float}
\newcommand{\s}{\mathcal{S}}
\newcommand{\GG}{\mathcal{G}}
\newcommand{\MMM}{\mathcal{M}}

\newcommand{\OO}{\mathcal{O}}

\newcommand{\LL}{\mathcal{L}}
\newcommand{\BB}{\mathcal{B}}
\newcommand{\NN}{\mathcal{N}}

\newcommand{\LLL}{ \tilde{\mathcal{L}} }
\newcommand{\BBB}{ \tilde{\mathcal{B}} }
\newcommand{\NNN}{ \tilde{\mathcal{N}} }

\newcommand{\BBBB}{ \hat{\mathcal{B}} }

\newcommand{\R}{\mathbb{R}}

\newcommand{\M}{$\mu$G}
\newcommand{\MM}{multi-$\mu$G}

\usepackage[utf8]{inputenc}
\usepackage[english]{babel}
\newtheorem{theorem}{Theorem}

\newtheorem{prop}{Proposition}

\makenomenclature
\usepackage{etoolbox}
\renewcommand\nomgroup[1]{%
  \item[\bfseries
  \ifstrequal{#1}{A}{Sets and Indices}{%
  \ifstrequal{#1}{P}{Parameters}{%
  \ifstrequal{#1}{V}{Variables}{}}}%
]}

\usepackage{cite}
\ifCLASSINFOpdf
\else
\fi
\usepackage{amsmath}
\usepackage{algorithmic}
\usepackage{array}
\begin{document}
%\title{Resilient Operation of Multi-Microgrids:\\ A Frequency-Constrained Approach}
\title{\textcolor{myblue}{Towards} Resilient Operation of Multi-Microgrids:\\ {\huge An MISOCP-Based Frequency-Constrained Approach}}
\author{Amin~Gholami,~\IEEEmembership{Student Member,~IEEE,}
        and~Xu~Andy~Sun,~\IEEEmembership{Senior Member,~IEEE}
\thanks{The authors are with the H. Milton Stewart School of Industrial and Systems Engineering, Georgia Institute of Technology, Atlanta, GA 30332 USA (e-mail: \texttt{a.gholami{@}gatech.edu}; \texttt{andy.sun{@}isye.gatech.edu}).}}
\maketitle

\begin{abstract}
%High penetration of distributed energy resources (DERs) is transforming the paradigm in power system operation. The ability to provide electricity to customers while the main grid is disrupted has introduced the concept of microgrids (\M s) with many challenges and opportunities. Emergency control of dangerous transients caused by the transition between the grid-connected and island modes in \M s is one of the main challenges in this context. To address this challenge, we develop a novel framework for real-time resilient operation of multi-microgrid (\MM) networks subsequent to scheduled and unscheduled islanding phenomena. The overall framework is based on a near real-time decision support tool as well as a real-time monitoring and control scheme. Our decision support tool is inspired by the frequency response of \M s to an islanding event. Accordingly, we formulate the frequency-constrained operation of \MM s, including optimal load shedding and network topology control with AC power flow, as a mixed-integer second order cone programming (MISOCP) problem. Moreover, we develop valid inequalities which are used in a cutting-plane approach to eliminate frequency violations. The results of the optimization problem are arranged in a look-up table to be implemented in the network in the real-time monitoring and control stage. Our numerical experiments further illustrate that the proposed emergency control scheme can successfully monitor, verify, and act to guarantee that the multi-\M ~network remains within the operational limits during post-islanding frequency dynamics. 
High penetration of distributed energy resources (DERs) is transforming the paradigm in power system operation. The ability to provide electricity to customers while the main grid is disrupted has introduced the concept of microgrids (\M s) with many challenges and opportunities. Emergency control of dangerous transients caused by the transition between the grid-connected and island modes in \M s is one of the main challenges in this context. To address this challenge, this paper proposes a comprehensive optimization and
real-time control framework for maintaining frequency stability of multi-\M ~networks under an islanding event and for achieving optimal load shedding and network topology control with AC power flow constraints. The paper also develops a strong mixed-integer second-order cone programming (MISOCP)-based reformulation and a cutting plane algorithm for scalable computation. We believe this is the first time in the literature that such a framework for multi-\M ~network control is proposed, and its effectiveness is demonstrated with extensive numerical experiments.
%Numerical experiments illustrate that the proposed emergency control scheme can successfully monitor, verify, and act to guarantee that the multi-\M ~network remains within the operational limits during post-islanding frequency dynamics. 
%We believe this is the first time in the literature such a framework is proposed and is demonstrated for its effectiveness.

%we develop a novel framework for real-time resilient operation of multi-microgrid (\MM) networks subsequent to scheduled and unscheduled islanding. The overall framework is based on a near real-time decision support tool as well as a real-time monitoring and control scheme. Our decision support tool is inspired by the frequency response of \M s to an islanding event. We formulate the frequency-constrained operation of \MM s, including optimal load shedding and network topology control with AC power flow, as a mixed-integer second order cone programming (MISOCP) problem. Moreover, we develop valid inequalities which are used in a cutting-plane approach to eliminate frequency violations. The results of the optimization problem are arranged in a look-up table to be implemented in the network in the real-time monitoring and control stage. Our numerical experiments further illustrate that the proposed emergency control scheme can successfully monitor, verify, and act to guarantee that the multi-\M ~network remains within the operational limits during post-islanding frequency dynamics. 
\end{abstract}


%--------------------------------------------
% \nomenclature[A]{$\BB$}{Set of microgrid (\M) buses}%
% \nomenclature[A]{$\BBB$}{Set of linking buses}%
% \nomenclature[A]{$\BBBB$}{Set of boundary buses}%
% % \nomenclature[A]{$\MMM$}{Set of microgrids}%
% \nomenclature[A]{$\GG$}{Set of distributed energy resources (DERs)}
% \nomenclature[A]{$m, k$}{\M ~index}%
% \nomenclature[A]{$i, j$}{\M ~ bus index}%
% \nomenclature[A]{$g$}{DER index}
% \nomenclature[A]{$\LL$}{Set of lines}
% \nomenclature[A]{$\LLL$}{Set of tie-lines}
% %--------------------------------------------
% \nomenclature[P]{${M}, {M'}$}{Sufficiently large positive numbers}

 \section{Introduction} \label{sec: Introduction}
 \IEEEPARstart{M}{icrogrids} (\M s), as building blocks of smart distribution grids, provide a unique infrastructure for integrating a wide range of distributed energy resources (DERs) with different static and dynamic characteristics. They are able to operate in island mode and energize a portion of the grid while the main grid is down. This islanding capability of \M s is highly beneficial for both customers and electric utilities, especially in areas with frequent electrical outages. Although dynamic islanding is one of the basic objectives of building a \M, IEEE Std. 929-2000 \cite{IEEE-Std-929-2000-2000} and IEEE Std. 1547.7-2013 \cite{IEEE-Std-1547-2013} mandate that DERs shall detect the unintentional island mode and cease to energize the grid within two seconds, mainly due to safety concerns as well as complying with conventional control/protection schemes. Operation of DERs during intentional islanding has also been under consideration for future revisions of IEEE Std. 1547. Based on the current practices and standards, blackouts in \M s seem inevitable in the event of islanding (especially an unscheduled islanding which may occur subsequent to detection of abnormal conditions at the interconnection(s)). 
 
Intuitively, the disconnection of DERs is not an ideal solution, particularly in a restructured environment where electric utilities compete to provide a more reliable service to customers. In this context, a recent draft standard for interoperability of DERs in 2017 has provided some guidance on scheduled and unscheduled islanding processes \cite{IEEEStd-2017-Draft}. This draft standard defines an intentional local island as any portion of the grid that is totally within the bounds of a local power grid (e.g., a \M), and further states that DERs may have to adjust several settings which shall be enabled only when the intentional island is isolated from the main grid. This standard calls for adaptive protection and control schemes to be used in such circumstances. Our paper is motivated by this need, and is aimed at providing a practical solution to the islanding process in modern distribution networks which are comprised of multiple \M s, referred to as multi-microgrid (\MM) networks.

In a similar vein, \cite{Mahat-2010-UFLS,Iravani-2005-microgrid,balaguer-2011-intentional-Islanding} acknowledge that the current practice of disconnecting DERs following a disturbance is no longer a reliable solution. Specifically, reference \cite{Mahat-2010-UFLS} proposes an  under frequency load shedding (UFLS) scheme to be used subsequent to islanding in a distribution system. This scheme sheds an optimal number of loads based on a set of criteria including frequency, rate of change of frequency, customers’ willingness to pay, and load histories. The authors in \cite{Iravani-2005-microgrid} investigate autonomous operation of a distribution system as an individual \M. The paper demonstrates the transient behavior of such a \M ~due to preplanned and unplanned islanding processes. The authors also emphasize that future studies should develop control strategies/algorithms for multiple electronically interfaced DERs to achieve optimum response in terms of stability. In \cite{balaguer-2011-intentional-Islanding}, a controller for distributed generation (DG) inverters is designed for both grid-connected and intentional islanding modes. Moreover, an islanding-detection algorithm is developed in order to switch between the two modes. 

On the other hand, the operation of \MM s has been studied in the literature from different perspectives, such as their on-line dynamic security assessment \cite{Xie-2015-onlineSecurity}, interactive control for guaranteed small signal stability \cite{Xie-2016-interactive}, transient stability assessment \cite{Xie-2016-transient}, electricity market operator design \cite{VincentPoor-2015-multimicrogridMarket}, hierarchical outage management \cite{farzin-2016-MultiMicrogridResilience}, and self-healing \cite{JianhuiWang-2016-selfhealing} to name a few.

In this paper, we propose a novel framework for the resilient operation of multi-\M ~networks after a scheduled or unscheduled islanding in a distribution system. The framework is strategically designed in two parts. In the first part, we develop a near real-time decision support tool which is used to determine the optimal reconfiguration of the multi-\M ~network, cooperation between \M s (sharing their resources), new operating point of dispatchable DERs, and emergency load curtailments (if necessary). The second part of the framework pertains to the real-time monitoring and control of multi-\M s based on the outcomes of the decision support tool. The present paper is a significant extension to our recent work \cite{Amin-2018-HICSS} on a single \M ~operation. Specifically, the main contributions of this paper are summarized below.

\begin{itemize}
\item We formulate the real-time resilient operation, including optimal power flow, optimal load shedding, and optimal topology reconfiguration, of a multi-\M ~network as a mixed-integer nonlinear programming (MINLP) problem. Then, we propose a mixed-integer second order cone programming (MISOCP) relaxation to this problem, which considerably improves the computational efficiency of our control framework and renders it scalable in practical systems.
\item We derive necessary constraints for keeping the nadir and steady state frequency of the network within the permissible ranges, and introduce a new reformulation for frequency limitation constraints. This reformulation implicitly guarantees the frequency stability of the network after dangerous transients such as islanding. %We adopt these constraints to maintain desired operations while preventing frequency instability from developing into devastating blackouts.
%\item We introduce a new reformulation for frequency limitation constraints. This reformulation implicitly guarantees the frequency stability of the network after dangerous transients such as islanding.
\item We develop a set of valid inequalities and a separation scheme for incorporating the frequency constraints in the operation of a multi-\M ~network, and based on that, we establish a cutting-plane approach to eliminate the frequency violations in a computationally effective way.
\end{itemize}

%Meanwhile, we have proposed an optimization-based load shedding scheme for a single \M in our recent work \cite{Amin-2018-HICSS}. The present paper is a significant extension
%We observed that our mixed-integer linear programming (MILP) model can find the best location of load curtailments and outperforms the conventional scheme in terms of load shedding cost, number of curtailed customers, and nadir frequency. As a follow-up, this paper presents a more refined use case to further improve the resilience of distribution networks.

The rest of our paper is organized as follows. Section \ref{sec: self healing multi microgrid} introduces a resilient multi-\M ~network and gives an overview of the proposed scheme. The frequency response of multi-\M s to an islanding process is discussed in Section \ref{sec: Frequency response of multi microgrid}. In Section \ref{sec: Self healing problem formulation}, a basic MINLP model for the real-time resilient operation of \MM s is presented. Section \ref{Sec: Solution Methodology} is devoted to solution methodology, including the MISOCP relaxation and cutting plane algorithm. Section \ref{Sec: Computational Experiments} exhibits the efficiency of the novel approach using an illustrative case study, and finally, the paper concludes with Section \ref{Sec: Conclusions}.

\section{Resilient Operation of Multi-\M s} \label{sec: self healing multi microgrid} 
\subsection{Structure of a Multi-\M  \, Network}
A distribution network may experience a scheduled islanding due to several reasons such as enhanced reliability, economic dispatch decisions for self-supply, pre-emptive action prior to inclement weather, etc. Moreover, unscheduled islanding happens subsequent to the detection of abnormal conditions at the interconnection(s) \cite{IEEEStd-2017-Draft}. In either case, the distribution system can be further partitioned into multiple \M s, thereby improving the resilience of the system. 
Fig. \ref{Fig: Multi Microgrid} depicts a distribution network under such circumstances. As can be seen in this example, the distribution network is composed of four \M s, where each \M ~is connected  to the rest of the system through the point of common coupling (PCC). Note that \M s in a \MM ~network are commonly integrated via voltage-source-converter-(VSC)-based interfaces at the PCC, and the behavior of each \M ~is characterized by the control scheme of its interface \cite{majumder-2010-VSC-based}. PCCs are commonly equipped with intelligent electronic devices (IEDs) with synchrophasor capability \cite{Xie-2015-onlineSecurity}. \textcolor{myblue}{A communication network connects the IEDs to the distribution management system (DMS). Note that the resilience of this communication infrastructure (notably during an unscheduled outage) is of paramount importance to operators’ situational awareness.}

\textcolor{myblue}{In Fig. \ref{Fig: Multi Microgrid}(a),}
a set of buses (white fill in the figure), namely linking buses, are not categorized to any \M . Additionally, the lines (dashed/dotted in Fig. \ref{Fig: Multi Microgrid}(a), \textcolor{myblue}{or equivalently $l_1$ to $l_5$ in Fig. \ref{Fig: Multi Microgrid}(b)}) between such buses, namely linking lines, are equipped with switching relays, enabling various configurations for the \MM ~network. This portion of the distribution network that consists of the linking buses and linking lines is called the \textit{linking grid}.
\textcolor{myblue}{Fig. \ref{Fig: Multi Microgrid}(b) illustrates the linking grid associated with the \MM ~network of Fig. \ref{Fig: Multi Microgrid}(a).}
Finally, the buses by which each \M ~is connected to the linking grid (gray fill in the figure) are called boundary buses.
\begin{figure}[t]
 \includegraphics*[width=3.00in, keepaspectratio=true]{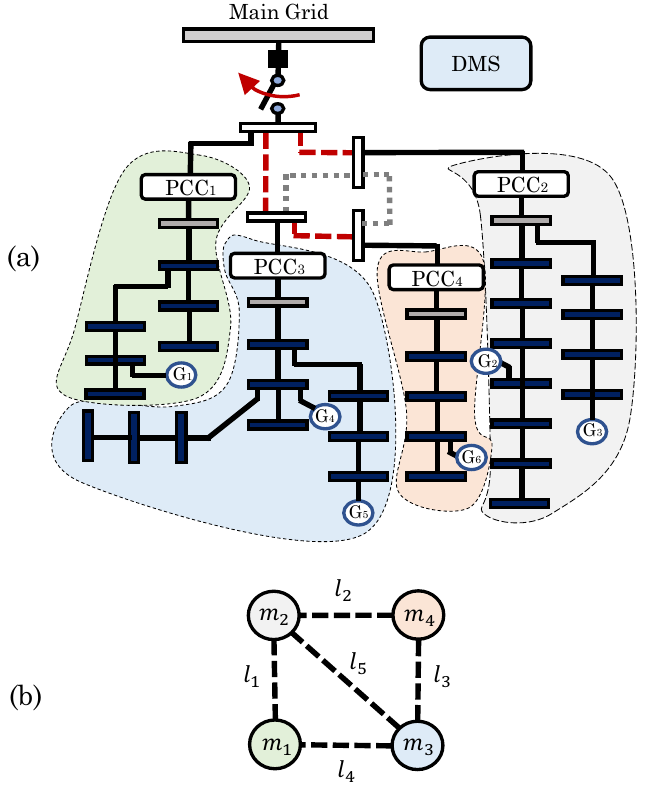}
 \centering
  \caption{ \textcolor{myblue}{Schematic diagram of a distribution system under islanding. (a) Multi-\M ~network. (b) Linking grid.} }
  \label{Fig: Multi Microgrid}
\end{figure}
\subsection{Overview of the Proposed Resilient Operation Scheme}
The general framework of the proposed resilience management scheme is illustrated in Fig. \ref{Fig: Two stage scheme}. This framework can be divided into two stages: i) near real-time decision support tool, and ii) real-time monitoring and control. In the first stage, the distribution system operator (DSO) leverages the state estimation (SE) module and obtains the input parameters of an optimization model. These data include the generation/consumption level of DERs/Loads, real and reactive power exchange at PCCs, and the status of the circuit breakers (i.e., network topology). Subsequently, the optimization model is solved and the following resilient operation strategies are determined: optimal configuration of the linking network, cooperation between \M s (sharing their DERs), new operating point of dispatchable DERs, and emergency load curtailments (if necessary). Note that the frequency limitations of the system are embedded in the optimization model to ensure the frequency stability of \MM s following the islanding event. In the next step, a look-up table is generated based on the results of the optimization model.
On the other side, in the second stage, the status of the main circuit breaker (i.e., the islanding status of the distribution network) is monitored using indication data. If an unscheduled/scheduled islanding happens, the pre-specified strategies will be implemented in the \MM ~network.

\textcolor{myblue}{The principal focus of this paper is on the first stage (left-hand side of Fig. \ref{Fig: Two stage scheme}), i.e., developing a near real-time decision support tool that will be thoroughly discussed in the following sections.
The second stage (right-hand side of Fig. \ref{Fig: Two stage scheme}) corresponds to the mechanisms for implementing such decisions. The details of these mechanisms, which are enabled by synchrophasor technology, go beyond the scope of this paper.}
% Detailed explanations about different parts of the proposed decision support tool are provided in the following sections
%
\begin{figure} [t]
\includegraphics*[width=3.4 in, keepaspectratio=true]{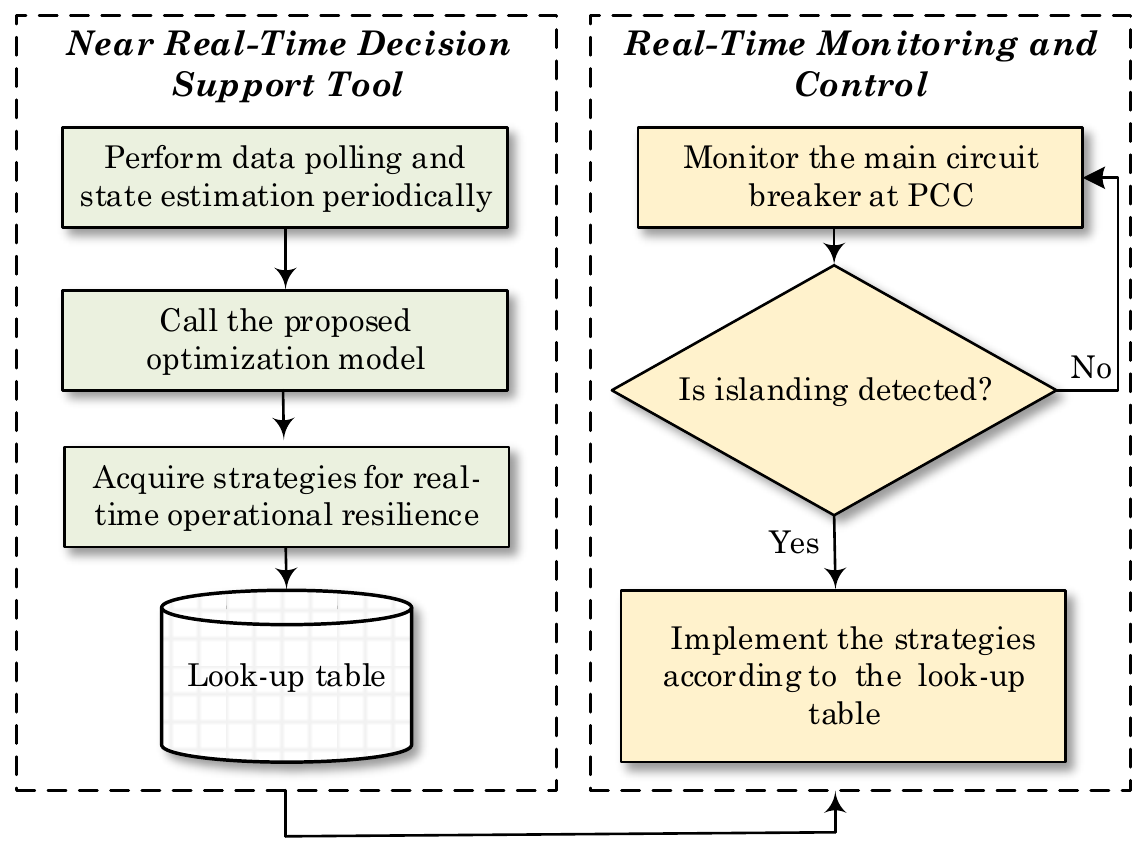}
    \centering
   \caption{The general framework of the proposed resilient operation approach.}
    \label{Fig: Two stage scheme}
\end{figure}
%
%\subsection{governing equations}
\section{Frequency Response of Multi-\M s Subsequent to Islanding} \label{sec: Frequency response of multi microgrid}
In this section, we will derive the steady-state and nadir frequencies of a \MM ~network subsequent to an imbalance between real power generation and consumption. Later in Section \ref{subsec: Frequency Constraints}, we will use these two metrics to construct our proposed frequency constraints, ensuring that they will remain in the permissible range during the transition between the grid-connected and island modes.
\subsection{Inertial Response}
 As mentioned earlier, \M s in a \MM ~network are integrated via VSC-based interfaces at the PCC. Meanwhile, VSC-based interfaces are controlled in such a way that they emulate the behavior of conventional synchronous machines \cite{Xie-2016-transient}. Inspired by this fact,
%we start with the frequency behavior of individual synchronous generators.
let \textcolor{myblue}{us} first focus on inertial response of \M s. Suppose $\MMM$ is the set of all \M s in the \MM ~network.
%\MM s by excluding the feedback loop in Fig. \ref{Fig: SFR model}.
The artificial swing equation describes the inertial frequency dynamics of each $m \in \MMM$, 
\begin{equation} \label{eq: individual swing equation}
\frac{d\Delta\omega_m}{dt}=\frac{1}{2H_m} \left( {\Delta P^M_m - \Delta P^E_m } \right),
\end{equation}
where $\Delta\omega_m$ is the frequency deviation in p.u.; $H_m$ is the artificial inertia constant in seconds; $\Delta P^M_m$ and $\Delta P^E_m$ are the mechanical and electrical power deviations in p.u., respectively. Based on (\ref{eq: individual swing equation}), modeling interconnected \M s can be realized by the so-called aggregation method \cite{Kundur-1994-Stability-Book}. Without loss of generality, we assume that for each $m \in \MMM$, equation (\ref{eq: individual swing equation}) is per-unitized based on a common power, $S_{Base}$. We define the center of inertia (COI) frequency as
\begin{equation} \label{eq: COI frequency}
\omega_{COI}:=\frac{\sum \limits_{m \in \MMM} H_m \omega_m}{\sum \limits_{m \in \MMM} H_m }.
\end{equation}
\begin{prop} The swing equation of a fictitious equivalent generator whose frequency is equal to $\omega_{COI}$ has the same form as 
	\begin{equation} \label{eq: equivalent swing equation}
	\frac{d\Delta\omega_{COI}}{dt}=\frac{1}{2H_a} \left( {\Delta P^M_a - \Delta P^E_a } \right),
	\end{equation}	
 	where $H_a$, $\Delta P^M_a$, and $\Delta P^E_a$ are defined below
	\begin{equation} \label{eq: COI Inertia}
	H_a:= \sum \limits_{m \in \MMM} H_m,
	\end{equation}
	\begin{equation} \label{eq: COI mech and elec power}
	\Delta P^M_a:= \sum \limits_{m \in \MMM}  \Delta P^M_m,\; \Delta P^E_a:= \sum \limits_{m \in \MMM}  \Delta P^E_m.
	\end{equation}
%		(\ref{eq: individual swing equation})}.
\end{prop}
%
%For a proof, reader can refer to \cite{ULBIG-ETH-Inertia}. 
\begin{proof}
%The proof can be found in the literature (see \cite{ULBIG-ETH-Inertia} for instance). However, we will provide it here to further clarify the motivations of our approach.
A complete proof of this basic result cannot be easily located in the literature. Therefore, we provide one here. 
Consider a small deviation from the initial value in (\ref{eq: COI frequency}), i.e., $\Delta \omega_{COI} := \omega_{COI} - \omega_{COI}^0$ and $\Delta \omega_m := \omega_m - \omega_m^0 $, and take derivative of its both sides with respect to $t$:
\begin{equation} \label{eq: proof of swing 1}
\frac{d\Delta\omega_{COI}}{dt}=\frac{\sum \limits_{m \in \MMM} H_m \frac{d\Delta\omega_m}{dt}}{\sum \limits_{m \in \MMM} H_m }.
\end{equation}
Then, re-arrange (\ref{eq: individual swing equation}) as
\begin{equation} \label{eq: proof of swing 2}
H_m \frac{d\Delta\omega_m}{dt}=\frac{1}{2} \left( {\Delta P^M_m - \Delta P^E_m } \right).
\end{equation}
Now substitute (\ref{eq: proof of swing 2}) in (\ref{eq: proof of swing 1}), as 
\begin{equation} \label{eq: proof of swing 3}
\frac{d\Delta\omega_{COI}}{dt}=\frac{ \sum \limits_{m \in \MMM} \frac{1}{2} \left( {\Delta P^M_m - \Delta P^E_m } \right) }{ \sum \limits_{m \in \MMM} H_m }.
\end{equation}
With the definition of \eqref{eq: COI Inertia}-\eqref{eq: COI mech and elec power}, we get \eqref{eq: equivalent swing equation}.
%Now define
%\begin{equation} \label{eq: COI Inertia}
%H_a:= \sum \limits_{m \in \MMM} H_m,
%\end{equation}
%%
%\begin{equation} \label{eq: COI mech and elec power}
%\Delta P^M_a:= \sum \limits_{m \in \MMM}  \Delta P^M_m,\; \Delta P^E_a:= \sum \limits_{m \in \MMM}  \Delta P^E_m.
%\end{equation}
%Therefore, the swing equation of the equivalent generator would be equal to
%\begin{equation} \label{eq: equivalent swing equation}
%\frac{d\Delta\omega_{COI}}{dt}=\frac{1}{2H_a} \left( {\Delta P^M_a - \Delta P^E_a } \right).
%\end{equation}
\end{proof}
%\vspace{-7pt}
In the rest of the paper, the COI frequency is simply denoted by $\omega$ instead of $\omega_{COI}$.
\vspace{-5pt}
\subsection{Droop Response}
Now we construct the aggregated system frequency response (SFR) model of a \MM ~network as depicted in Fig. \ref{Fig: Aggregated SFR model}. In this model, the transfer function $\frac{1}{{2H_a s + D}}$ in the forward path represents the swing equation (\ref{eq: equivalent swing equation}) as well as the frequency-dependent behavior of the loads which is lumped into a single damping constant $D$. \textcolor{myblue}{In this paper, this damping constant $D$ is assumed to remain unchanged while aggregating different \M s.}
Different feedback loops in Fig. \ref{Fig: Aggregated SFR model} model the contribution of each \M ~to the droop control of the \MM ~network \cite{Kundur-1994-Stability-Book}. For each $m \in \MMM$, $R_m$ is the droop constant of the VSC; $T_m$ and $T'_m$ are the corresponding time constants.
\begin{figure} [h] 
\tikzstyle{block} = [draw, fill=white!20, rectangle, line width=0.5mm, 
    minimum height=2.1em, minimum width=3.9em]
    
\tikzstyle{block1} = [draw, fill=white!20, rectangle, line width=0.5mm,
        minimum height=2.1em, minimum width=2.1em]
  
\tikzstyle{sum} = [draw, fill=blue!20, circle, node distance=1.2cm,line width=0.4mm]
\tikzstyle{input} = [coordinate]
\tikzstyle{output} = [coordinate]
\tikzstyle{pinstyle} = [pin edge={to-,thin,black}]

\begin{tikzpicture}[auto, node distance=1.1cm,>=latex']
    \node [input, name=input] {};
    \node [sum, right of=input] (sum) {};
    \node [block, right of=sum, node distance=3.1cm] (controller) {$\frac{1}{{2H_{a}s + D}}$};
    \node [output, right of=controller, node distance=4.5cm] (output) {};
    \coordinate [below of=controller] (tmp);
    
    \node [block, below of=controller] (measurements) {$\frac{{ {1 + {T_1}s}}}{{1 +      T^{\prime}_1 s}}$};
    \node [block1, right of=measurements, node distance=1.8cm] (measurements2)             {$\frac{1}{R_1}$};

%    \node [block, left of=measurements, node distance=1.8cm] (perunit) {$\frac{S^B_1}{\sum \limits_{m \in \MMM} S^B_m}$};

      \node [block, below of=measurements] (measurements3) {$\frac{{ {1 + {T_m}s}}}{{1     + T^{\prime}_m s}}$};
      
      \node [block1, right of=measurements3, node distance=1.8cm] (measurements4)           {$\frac{1}{R_m}$};
      
%      \node [block, left of=measurements3, node distance=1.8cm] (perunit2)                 {$\frac{S^B_m}{\sum \limits_{m \in \MMM} S^B_m}$};

    \draw [draw,->, line width=0.4mm] (input) -- node {$\Delta P_{a}(s)$} (sum);
    \draw [->,line width=0.4mm] (sum) -- node {} (controller);
    \draw [->,line width=0.4mm] (controller) -- node [name=y] {$\qquad \qquad \qquad \Delta \omega (s)$}(output);
    
    \draw [->,line width=0.4mm] (y) |- (measurements2);
    \draw [draw,->,line width=0.4mm] (measurements2) -- node {} (measurements);
%    \draw [draw,->,line width=0.4mm] (measurements) -- node {} (perunit);
  
      \draw [dashed] [->,line width=0.3mm] (y) |- (measurements4);
      \draw [dashed] [draw,->,line width=0.3mm] (measurements4) -- node {} (measurements3);
%      \draw [dashed] [draw,->,line width=0.3mm] (measurements3) -- node {} (perunit2);
  
    \draw [->,line width=0.4mm] (measurements) -| node[pos=0.89] {$-$} 
           node [near end] {} (sum);
           
           \draw [dashed][->,line width=0.3mm] (measurements3) -| node[pos=0.89] {$ \: - $} 
           node [near end] {} (sum.south west);
           
\end{tikzpicture}
\centering
\caption{Block diagram of the aggregated SFR model.}
\label{Fig: Aggregated SFR model}
\end{figure}
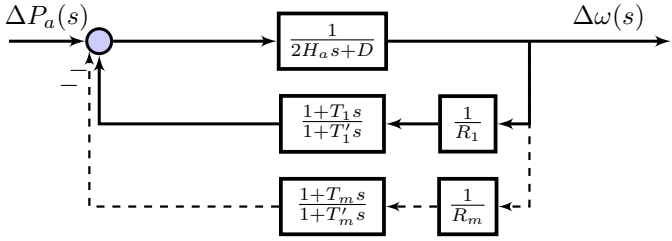

In general, the order of this SFR model is $|\MMM| + 1$. In particular, however, we are interested in the steady state and nadir outputs of the SFR model. It can be shown that the steady state output of this general-order model is not a function of the time constant  $T'_m$. Moreover, the results of a sensitivity analysis on the parameters of a similar SFR model confirms that the nadir frequency is less sensitive to $T'_m$ \cite{Ahmadi-2014-UnitCommitment}. Accordingly, we assume identical values of $T'_m$ for all \M s in the \MM ~network, i.e., $T':=T'_m, \forall m \in \MMM$. Consequently, the transfer function of the aggregated SFR model can be written as (\ref{eq: SFR Transfer Function}), with the additional parameters defined in (\ref{eq: transfer function para definitions}):
\begin{equation} \label{eq: SFR Transfer Function}
\mathcal{H}(s) = \frac{1 + T's} {2 H_a T' \left( s^2 + 2 \xi \omega_n s + \omega _n^2 \right)},
\end{equation}
\begin{subequations} \label{eq: transfer function para definitions}
	\begin{align}
	&{\omega _n} := \sqrt {\frac{D + 1/R_a}{2 H_a T'}},\; \xi  := \frac{ 2H_a + T'D + K_a } {2 \sqrt {2 H_a T'\left( D + 1/R_a \right)}},\\
	& \frac{1}{R_a}:=\sum \limits_{m \in \MMM} \frac{1}{R_m}, K_a := \sum \limits_{m \in \MMM} \frac{T_m}{R_m}, \label{eq: aggregated Ra and Ta}
	\end{align}
\end{subequations}
where $\Delta P_a(s)$ is the disturbance power in the \MM.

\vspace{-7pt}
\subsection{Steady State and Nadir Frequencies at COI}
In general, the dynamic behavior of the aggregated SFR model can be described by two parameters $\xi$ and ${\omega _n}$. If $\xi  = 0$, we will have an oscillatory system where the transient response will not die out. If $\xi \in (0,1)$, the transient frequency response is oscillatory (under-damped). When $\xi  = 1$, we are in the critically-damped condition, and finally, if $\xi  \in (1,+\infty)$, the frequency response will be over-damped. We shall now analyze the frequency response of the system to the unit-step input, i.e., $\Delta P_a(s)=1/s$ for three cases: the under-damped, critically-damped, and over-damped cases.
\begin{prop} \label{Prop: Nadir and steady state freq} {In the under-damped case, the steady state and nadir COI frequencies of a \MM ~network after a unit-step disturbance can be obtained by (\ref{eq: steady state frequency}) and (\ref{eq: Nadir Frequency}), respectively, i.e.,
\begin{equation} \label{eq: steady state frequency}
\Delta \omega(t^{ss}) = \frac{1}{D+1/R_a},
\end{equation}
\begin{equation} \label{eq: Nadir Frequency}
\begin{aligned}
\Delta \omega \left( t^{N} \right) = \frac{1} {D + 1/R_a} \left( {1 + \sqrt {\frac{T'-R_a K_a} {2H_a R_a}} {e^{ - \xi {\omega _n}{t^{N}}}}} \right),
\end{aligned}
\end{equation}
where $t^N$ in (\ref{eq: Nadir Frequency}) can be calculated as follows:
\begin{equation} \label{eq: Nadir time}
\begin{aligned}
t^{N} &= 
\begin{cases}
\frac{1} {\omega _r} \left( {\pi  - {\tan}^{ - 1} \left( {\frac{{{\omega_r}T'}}{{1 - \xi {\omega _n}T'}}} \right)} \right), & \text{if } \xi {\omega_n}T' < 1, \\
\frac{\pi}{2 \omega _r}, & \text{if } \xi {\omega_n}T' = 1, \\
\frac{1} {\omega_r} \left( { {{\tan }^{ - 1}} \left( {\frac{{{\omega _r}T'}}{{\xi {\omega _n}T'-1}}} \right)} \right), & \text{if } \xi {\omega_n}T' > 1.
\end{cases} 
\end{aligned}
\end{equation}
Additionally, in the critically-damped and over-damped cases, the nadir COI frequency is equal to the steady state COI frequency, and both can be calculated according to (\ref{eq: steady state frequency}). }
\end{prop}
% WE removed this in the revised manuscript
%%%%%%%%%%%%%%%%%%%%%%%%%%%%%%%%%%%%%%%%%%%%%%%%%%%%%%%%%%%%%%%%%%%%%%%%%%%%%%%%%%%%%%%%%%%%%%%%%
\begin{proof}
In the under-damped case, the poles of the system are $s_{1,2} = - \xi \omega_n \pm j\omega_r$, where $ \omega _r = \omega _n \sqrt {1 - \xi ^2}$ is the damped natural frequency and $j=\sqrt{-1}$ is the imaginary unit. In this case, the unit-step response is
\begin{equation} \label{eq: Step Response}
\begin{aligned}
\Delta \omega(t) = \frac{1} {2 H_a T'} \Big( \frac{1}{{\omega _n^2}} & + \frac{e^{ - \xi {\omega _n}t}}{\omega_r} \Big( T'\sin ( {{\omega_r} t} ) \\
& - \frac{1}{\omega_n} \sin ( \omega_r t + \phi  ) \Big) \Big).
\end{aligned}
\end{equation}
where $ \phi := \tan^{-1} ( \frac{\sqrt {1 - \xi ^2} }{\xi } )$. By definition, the steady state frequency is equal to $\Delta \omega(t^{ss}) := \mathop {\lim }\limits_{t \to +\infty } \Delta \omega \left( t \right)$, which leads to (\ref{eq: steady state frequency}).
% \begin{equation} \label{eq: steady state frequency}
% \Delta \omega(t^{ss}) = \mathop {\lim }\limits_{t \to +\infty } \Delta \omega \left( t \right) = \frac{1}{D+1/R_a}.
% \end{equation}
The time when the frequency nadir happens (when the lowest frequency is reached before the frequency starts to recover) can be calculated by solving the optimization problem $t^N :=\min \{ t:\frac{d\Delta \omega \left( t \right)} {dt} = 0, t \in \R _{++} \}$. The closed-form solution to this problem is equal to (\ref{eq: Nadir time}). Additionally. substitution of $t^N$ in (\ref{eq: Step Response}) yields (\ref{eq: Nadir Frequency}).
% Accordingly, the nadir frequency is equal to
% \begin{equation} \label{eq: Nadir Frequency}
% \begin{aligned}
% \Delta \omega \left( t^{N} \right) = \frac{1} {D + 1/R_a} \left( {1 + \sqrt {\frac{T'-R_a K_a} {2H_a R_a}} {e^{ - \xi {\omega _n}{t^{N}}}}} \right),
% \end{aligned}
% \end{equation}
% where $t^N$ can be calculated as follows:
% \begin{equation} \label{eq: Nadir time}
% \begin{aligned}
% t^{N} &= 
% \begin{cases}
% \frac{1} {\omega _r} \left( {\pi  - {\tan}^{ - 1} \left( {\frac{{{\omega_r}T'}}{{1 - \xi {\omega _n}T'}}} \right)} \right), & \text{if } \xi {\omega_n}T' < 1, \\
% \frac{\pi}{2 \omega _r}, & \text{if } \xi {\omega_n}T' = 1, \\
% \frac{1} {\omega_r} \left( { {{\tan }^{ - 1}} \left( {\frac{{{\omega _r}T'}}{{\xi {\omega _n}T'-1}}} \right)} \right), & \text{if } \xi {\omega_n}T' > 1.
% \end{cases} 
% \end{aligned}
% \end{equation}
Observe that when the two poles of the transfer function (\ref{eq: SFR Transfer Function}) are nearly equal, i.e., $s_{1,2} = -  \omega_n$, the system is approximated by a critically-damped one. Moreover, in the over-damped case, the two poles  of the transfer function are negative real and unequal, i.e., $s_{1,2} = - \xi \omega_n \pm \omega _n \sqrt {\xi ^2 -1}$. In the last two cases, no overshoot or undershoot is observed in the transient response of the system, and consequently, the nadir frequency is equal to the steady state frequency which is identical to (\ref{eq: steady state frequency}).
\end{proof}
\vspace{-5pt}
%%%%%%%%%%%%%%%%%%%%%%%%%%%%%%%%%%%%%%%%%%%%%%%%%%%%%%%%%%%%%%%%%%%%%%%%%%%%%%%%%%%%%%%%%%%%%
% Remark: In the literature (e.g. \cite{Ahmadi-2014-UnitCommitment}) only partial result of $\xi\omega_nT'>1$ is given. However, all three cases given in \eqref{eq: Nadir time} are possible and important to consider.  
\textcolor{myblue}{
% The proof of Proposition \ref{Prop: Nadir and steady state freq} is quite straightforward and therefore omitted here because of space limit.
The interested reader is referred to Proposition 3 in
\cite{Paganini-2017-InertiaProof} for similar results under different settings}. 
\textcolor{myblue}{Now we are ready to adopt the steady-state and nadir frequencies at COI in order to build our optimization model for the resilient operation of a \MM ~network.}
\vspace{-7pt}
\section{Resilient Operation Problem Formulation} \label{sec: Self healing problem formulation}
Consider a linking grid $\NNN=(\BBB,\LLL)$, where $\BBB$ and $\LLL$ denote the set of linking buses and linking lines, respectively. We assume that the distribution network under study is comprised of a set of \M s, i.e., $m \in \MMM$, where each \M ~is modeled as a disjoint network $\NN_m=(\BB_m,\LL_m)$. Without loss of generality, we assume only one PCC for each \M, and the corresponding boundary bus is denoted by $\BBBB$. In this section, we aim to introduce an optimization model which is able to determine the optimal resilience improvement strategy, including optimal load shedding and network topology control with AC power flow, in the wake of a scheduled/unscheduled islanding in a \MM ~network. Our model is formulated as follows.
%With this aim in mind, we categorize the objective function and constraints of such into different sets, and describe them
% With this aim in mind, we first establish the optimization model without the frequency constraints, and then, the frequency-constrained formulation is developed. 
\vspace{-7pt}
\subsection{Objective Function}
The objective function (\ref{eq: NLP Obj}) is to minimize the total load shedding cost in all \M s:
\begin{equation} \label{eq: NLP Obj}
\min \sum \limits_{m \in \MMM} \sum \limits_{i \in \BB _m } {\lambda _{mi}^{VOLL} \left( {1 - x_{mi}} \right) \bar p_{mi}^D}, 
\end{equation}
where $\lambda _{mi}^{VOLL}$ is the value of lost load (VOLL) in \M ~$m$ and bus $i$; $\bar p_{mi}^D$ is the pre-islanding active power consumption obtained from state estimation (SE); and $x_{mi}$ is a binary variable indicating the status of such a load after islanding happens.
\vspace{-5pt}
\subsection{Real-Time AC Power Flow Limitations in \M s}
The set of constraints (\ref{eq: active power balance})-(\ref{eq: reactive load model}) \textcolor{myblue}{which are defined for each $m \in \MMM$} guarantee the AC power flow security of each \M ~after the islanding event. Let $G_{mij}$ and $B_{mij}$ be the conductance and susceptance of line $(i,j)$ in \M ~$m$; and $f_{mij}^P$ and $f_{mij}^Q$ be the active and reactive flow of that line. Additionally, let $p_{mg}^G$ and $q_{mg}^G$ be the active and reactive power output of DER $g$ in \M ~$m$; and similarly, $p_{mi}^D$ and $q_{mi}^D$ be the active and reactive power consumption of the load at bus $i$ in \M ~$m$. We define $V_{mi}$ and $\theta_{mi}$ as the voltage magnitude and angle of bus $i$ in \M ~$m$. Finally, $\Delta {P_m}$ and $\Delta {Q_m}$ denote the active and reactive power exchange between the \M ~$m$ and the linking grid (through the VSC). Based on this notation, constraints (\ref{eq: active power balance}) and (\ref{eq: reactive power balance}) model the active and reactive power balance within each \M. Similarly, constraints (\ref{eq: boundary active power balance}) and (\ref{eq: boundary reactive power balance}) are related to the active and reactive power balance at the boundary buses. Note that $\OO$ in these equations is the mapping of the set of DERs into the set of buses. The set of equations (\ref{eq: internal active power flow})-(\ref{eq: voltage limit}) constitute the AC power flow equations, line flow limits, and voltage bounds in each \M. Finally, active and reactive power demands at different buses are modeled by the voltage-dependent ZIP model (\ref{eq: active load model}) and (\ref{eq: reactive load model}), where $\kappa ^{PI}$, $\kappa ^{PC}$, and $\kappa ^{PP}$ denote the coefficients of constant impedance, constant current, and constant power terms in active power loads, respectively. These coefficients are defined in the same way for reactive power loads.
% %--------------------------------------------
\begin{equation} \label{eq: active power balance}
\begin{aligned}
\sum \limits_{g:(g,i) \in {\OO}_m} p_{mg}^G - x_{mi} p_{mi}^D = \sum \limits_{\left( {i,j} \right) \in \LL_m }  f_{mij}^P, \forall i \in \BB_m
\end{aligned}
\end{equation}
\begin{equation} \label{eq: reactive power balance}
\begin{aligned}
& \sum\limits_{g:(g,i) \in  {\OO}_m} q_{mg}^G - x_{mi} q_{mi}^D = \sum \limits_{\left( {i,j} \right) \in \LL_m } f_{mij}^Q, \forall i \in \BB_m
\end{aligned}
\end{equation}
%-------------------------------------------
\begin{equation} \label{eq: boundary active power balance}
\begin{aligned}
\sum \limits_{g:(g,i) \in {\OO}_m} {p_{mg}^G} & - {x_{mi}}p_{mi}^D + \Delta {P_m} = \sum \limits_{\left( {i,j} \right) \in \LL_m } f_{mij}^P, 
\forall i \in \BBBB_m
\end{aligned}
\end{equation}
\begin{equation} \label{eq: boundary reactive power balance}
\begin{aligned}
\sum \limits_{g:(g,i) \in {\OO}_m} {q_{mg}^G} & - {x_{mi}}q_{mi}^D + \Delta {Q_m}= {\rm{ }}\sum\limits_{\left( {i,j} \right) \in \LL_m } f_{mij}^Q, 
 \forall i \in \BBBB_m
\end{aligned}
\end{equation}
\vspace{-3pt}
\begin{equation} \label{eq: internal active power flow}
\begin{aligned}
& f_{mij}^P =  G_{mij} \Big( V_{mi}^2 - V_{mi} V_{mj} \cos \left( \theta_{mi} - \theta_{mj} \right) \Big) \\
& - B_{mij} V_{mi} V_{mj} \sin \left( \theta_{mi} - \theta_{mj} \right), \: \forall \left( {i,j} \right) \in  \LL_m
\end{aligned}
\end{equation}
\begin{equation} \label{eq: internal reactive power flow}
\begin{aligned}
& f_{mij}^Q = - B_{mij} \Big( V_{mi}^2 - V_{mi} V_{mj} \cos \left( \theta_{mi} - \theta_{mj} \right) \Big) \\
& - G_{mij} V_{mi} V_{mj} \sin \left( \theta_{mi} - \theta_{mj} \right),\forall \left( {i,j} \right) \in  \LL_m
\end{aligned}
\end{equation}
\begin{equation} \label{eq: internal active flow limit loss}
\begin{aligned}
f_{mij}^P + f_{mji}^P \le f_{mij}^{\max},\: \forall (i,j) \in \LL_m
\end{aligned}
\end{equation}
\begin{equation} \label{eq: voltage limit}
\begin{aligned}
& V_{mi}^{\min} \le V_{mi} \le V_{mi}^{\max}, \: \forall i \in (\BB_m \cup \BBBB_m)
\end{aligned}
\end{equation}
\vspace{-5pt}
\begin{equation} \label{eq: active load model}
\begin{aligned}
p_{mi}^D = & \overline p_{mi}^D \Big( \kappa_{mi}^{PI} V_{mi}^2 + \kappa_{mi}^{PC} V_{mi} + \kappa_{mi}^{PP} \Big), 
 \forall i \in (\BB_m \cup \BBBB_m)
\end{aligned}
\end{equation}
\vspace{-5pt}
\begin{equation} \label{eq: reactive load model}
\begin{aligned}
q_{mi}^D = & \overline q_{mi}^D \Big( \kappa_{mi}^{QI} V_{mi}^2 + \kappa_{mi}^{QC} V_{mi} + \kappa_{mi}^{QP} \Big), \forall i \in (\BB_m \cup \BBBB_m).
\end{aligned}
\end{equation}
\vspace{-12pt}
\subsection{Real-Time AC Power Flow Limitations in the Linking Grid}
Similarly, this group of constraints are associated with the AC power flow limitations of the linking grid. Here, line switching is available, therefore, $Z_{mk}$ is a binary variable indicating the status of the linking line $(m,k)$. 
\textcolor{myblue}{It is worth mentioning that connection/disconnection of \M s to the linking grid is performed through the switchgear located at PCCs and line switching in the linking grid is commonly available through the distribution automation switches and isolators \cite{Xie-2015-onlineSecurity}.}
Let $M_{mk}$ be a sufficiently large positive number. In these constraints, in terms of notation, we use tilde over the variables and parameters to make the difference between the linking grid and the rest of the distribution grid. In particular, equations (\ref{eq: external active power balance}) and (\ref{eq: external reactive power balance}) model the active and reactive power balance at external buses. The group of constraints (\ref{eq: external active power flow up})-(\ref{eq: external voltage limit}) are associated with the AC power flow equations (where the lines are allowed to be switched on and off), line flow limits, and voltage bounds in the linking grid. 
\begin{equation} \label{eq: external active power balance}
\begin{aligned}
- \Delta {P_m} = \sum \limits_{(m,k) \in \LLL } \tilde{f}_{mk}^P, \: \forall m \in \MMM
\end{aligned}
\end{equation}
\begin{equation} \label{eq: external reactive power balance}
\begin{aligned}
- \Delta {Q_m} = \sum \limits_{(m,k) \in \LLL } \tilde{f}_{mk}^{Q}, \: \forall m \in \MMM
\end{aligned}
\end{equation}
\begin{equation} \label{eq: external active power flow up}
\begin{aligned}
- \tilde{f}_{mk}^{\textcolor{myblue}{P}} &+ \tilde{G}_{mk} \Big( \tilde{V}_m^2 - \tilde{V}_m \tilde{V}_k \cos( \theta_m - \theta_k ) \Big)\\
&- \tilde{B}_{mk} \tilde{V}_m \tilde{V}_k \sin ( \theta_m - \theta_k ) \\
&+ \left( 1 - Z_{mk} \right) M_{mk} \ge 0, \: \forall (m,k) \in \LLL
\end{aligned}
\end{equation}
\begin{equation} \label{eq: external active power flow down}
\begin{aligned}
- \tilde{f}_{mk}^P & +  \tilde{G}_{mk} \Big( \tilde{V}_m^2 - \tilde{V}_m \tilde{V}_k \cos( \theta_m - \theta_k ) \Big)\\
&- \tilde{B}_{mk} \tilde{V}_m \tilde{V}_k \sin ( \theta_m - \theta_k ) \\
&- \left( 1 - Z_{mk} \right) M_{mk} \le 0, \: \forall \left( {m,k} \right) \in \LLL
\end{aligned}
\end{equation}
\begin{equation} \label{eq: external reactive power flow up}
\begin{aligned}
- \tilde{f}_{mk}^Q & -  \tilde{B}_{mk} \Big( V_m^2 - \tilde{V}_m \tilde{V}_k \cos( \theta_m - \theta_k ) \Big) \\
& - \tilde{G}_{mk} \tilde{V}_m \tilde{V}_k \sin ( \theta_m - \theta_k ) \\
& + \left( 1 - Z_{mk} \right) M'_{mk} \ge 0, \: \forall \left( {m,k} \right) \in \LLL
\end{aligned}
\end{equation}
\begin{equation} \label{eq: external reactive power flow down}
\begin{aligned}
 - \tilde{f}_{mk}^Q & -  \tilde{B}_{mk} \Big( V_m^2 - \tilde{V}_m \tilde{V}_k \cos( \theta_m - \theta_k ) \Big) \\
 & - \tilde{G}_{mk} \tilde{V}_m \tilde{V}_k \sin ( \theta_m - \theta_k ) \\
 & - \left( {1 - {Z_{mk} }} \right) M'_{mk} \le 0,  \: \forall \left( {m,k} \right) \in \LLL
\end{aligned}
\end{equation}
\begin{equation} \label{eq: external active flow limit}
- \tilde{f}_{mk}^{P,\max} Z_{mk} \le \tilde{f}_{mk}^P \le \tilde{f}_{mk}^{P,\max} Z_{mk}, \: \forall \left( {m,k} \right) \in \LLL
\end{equation}
\begin{equation} \label{eq: external reactive flow limit}
- \tilde{f}_{mk}^{Q,\max} Z_{mk} \le \tilde{f}_{mk}^Q \le \tilde{f}_{mk}^{Q,\max} Z_{mk}, \: \forall \left( {m,k} \right) \in \LLL
\end{equation}
\begin{equation} \label{eq: external active flow limit loss}
\begin{aligned}
\tilde{f}_{mk}^P + \tilde{f}_{\textcolor{myblue}{{km}}}^P \le \tilde{f}_{mk}^{P,Loss,\max}, \: \forall \left( {m,k} \right) \in \LLL
\end{aligned}
\end{equation}
\begin{equation} \label{eq: external voltage limit}
\tilde{V}_{m}^{\min} \le \tilde{V}_{m} \le \tilde{V}_{m}^{\max}, \: \forall i \in \BBB, \: m \in \MMM.
\end{equation}
\subsection{DER Output Limitations and Binary Variable Declaration}
Finally, (\ref{eq: generator ramp limits})-(\ref{eq: binary variables}) pertain to the limitations on the output of the generators and the declaration of binary variables. In these constraints, $R^D$, $R^U$, and ${p^{G,0}}$ are the ramp-down, ramp-up, and pre-islanding active power generation of DERs, respectively.
\begin{equation} \label{eq: generator ramp limits}
\begin{aligned}
- & R_{mg}^D \le p_{mg}^G - p_{mg}^{G,0} \le R_{mg}^U, \: \forall g  \in \GG_m, m \in \MMM
\end{aligned}
\end{equation}
\begin{equation} \label{eq: generator active limits}
\begin{aligned}
& p_{mg}^{G,\min } \le p_{mg}^G \le p_{mg}^{G,\max }, \: \forall g  \in \GG_m, m \in \MMM
\end{aligned}
\end{equation}
\begin{equation} \label{eq: generator reactive limits}
\begin{aligned}
& q_{mg}^{G,\min } \le q_{mg}^G \le q_{mg}^{G,\max }, \: \forall g  \in \GG_m, m \in \MMM
\end{aligned}
\end{equation}
\begin{equation} \label{eq: binary variables}
\begin{aligned}
&x \in \left\{0,1\right\}^{ |\MMM| \times |\BB \cup \BBBB | }, \: Z \in \left\{0,1\right\}^{ |\LLL| }.
\end{aligned}
\end{equation}
\subsection{Frequency Constraints and Reformulation} \label{subsec: Frequency Constraints}
In Section \ref{sec: Frequency response of multi microgrid}, we developed the steady-state and nadir frequencies of a \MM ~network subsequent to an imbalance between real power generation and consumption. Indeed, these are two important metrics which are employed to ensure the frequency security of the network. Therefore, we aim to keep these two metrics within the permissible range while the \MM ~network moves from the grid-connected mode to the island mode. 
Note that subsequent to the islanding process, the distribution network might be partitioned into different components (each component might include one or more \M s), and the frequency security limitations must be met for each component separately. We propose the following constraints for ensuring the frequency security of the \MM ~network for each $\s \subseteq \NNN, \s  \ne \emptyset$:
\begin{subequations} \label{eq: General Cut}
\begin{align}
& \notag \Delta \omega_N^{\min} \le \alpha_{\s} \sum \limits_{m \in \BBB _{\s}  } { \Big( { - \Delta P_m^0 + \sum \limits_{i \in \BB _m } {\left( {1 - {x_{mi}}} \right)p_{mi}^D} } \Big) } \\ 
& + \mathcal{I}_M \left( \textnormal{$\s$ is connected} \right) + \mathcal{I}_M \left( \textnormal{$\s$ is isolated} \right), \label{eq: General Cut: nadir min} \\   
% & \text{and} \notag\\
%
& \notag  \textcolor{myblue}{    \Delta \omega_N^{\max} \ge \alpha_{\s} \sum \limits_{m \in \BBB _{\s}  } { \Big( { - \Delta P_m^0 + \sum \limits_{i \in \BB _m } {\left( {1 - {x_{mi}}} \right)p_{mi}^D} } \Big) }   }\\
& \textcolor{myblue}{  - \mathcal{I}_M \left( \textnormal{$\s$ is connected} \right) - \mathcal{I}_M \left( \textnormal{$\s$ is isolated} \right), }  \label{eq: General Cut: nadir max} \\
& \notag \Delta \omega_{ss}^{\min} \le \beta_{\s} \sum \limits_{m \in \BBB _{\s}  } { \Big( { - \Delta P_m^0 + \sum \limits_{i \in \BB _m } {\left( {1 - {x_{mi}}} \right)p_{mi}^D} } \Big) } \\
& + \mathcal{I}_M \left( \textnormal{$\s$ is connected} \right) + \mathcal{I}_M \left( \textnormal{$\s$ is isolated} \right), \label{eq: General Cut: steady state min} \\
& \notag  \textcolor{myblue}{  \Delta \omega_{ss}^{\max} \ge \beta_{\s} \sum \limits_{m \in \BBB _{\s}  } { \Big( { - \Delta P_m^0 + \sum \limits_{i \in \BB _m } {\left( {1 - {x_{mi}}} \right)p_{mi}^D} } \Big) }  }  \\
& \textcolor{myblue}{  - \mathcal{I}_M \left( \textnormal{$\s$ is connected} \right) - \mathcal{I}_M \left( \textnormal{$\s$ is isolated} \right), }  \label{eq: General Cut: steady state max}
\end{align}
\end{subequations}
where $\mathcal{I}_M$ is the indicator function whose value is equal to $0$ if the condition is satisfied, and equal to a sufficiently large number, otherwise. Moreover, $\alpha_{\s}$ and $\beta_{\s}$ are the nadir and steady state values of the unit-step frequency response, which are calculated in (\ref{eq: Nadir Frequency}) and (\ref{eq: steady state frequency}), respectively. The use of subscript $\s$ in these two parameters emphasizes that they should be calculated for each $\s \subseteq \NNN$, that is, the associated parameters $H_a$, $R_a$, and $K_a$ are obtained by (\ref{eq: COI Inertia}) and (\ref{eq: aggregated Ra and Ta}), where $m \in \MMM$ is replaced by $m \in \BBB _{\s}$. 
\textcolor{myblue}{ Note that $\Delta \omega_N^{\min}$/$\Delta \omega_N^{\max}$ and $\Delta \omega_{ss}^{\min}$/$\Delta \omega_{ss}^{\max}$ denote the lower/upper bound on the nadir and steady state frequencies, respectively.} Moreover, $\Delta P_m^0$ denotes the pre-islanding power exchange between \M ~$m$ and the linking grid. 
In (\ref{eq: General Cut}), the first term on the right-hand side of the inequities is indeed the multiplication of the unit-step response by the post-islanding net power mismatch (i.e., pre-islanding power exchange minus the amount of post-islanding load shedding). 
% It is worth mentioning that $\s$ is a function of the matrix $\mathbf{Z}$, i.e., the topology of the linking grid. 
Let us further investigate these frequency security constraints by defining
\begin{subequations}
\begin{align}
& \LLL(\s) := \{(m,k) \in \LLL: \: m,k \in \BBB _{\s}, \: m > k \}, \label{eq: edges of subgraph} \\
& \delta (\s) := \{(m,k) \in \LLL: \: m \in \BBB _{\s}, \: k \notin \BBB _{\s}, \: m > k \}. \label{eq: cutset of subgraph} 
\end{align}
\end{subequations}
\textcolor{myblue}{
Given a subgraph $\mathcal{S}$ of $\tilde{\mathcal{N}}$, $\LLL(\s)$ in (\ref{eq: edges of subgraph}) denotes the set of edges in the subgraph $\mathcal{S}$, i.e., the set of edges in $\tilde{\mathcal{L}}$ whose both ends are in ${\tilde{\mathcal{B}}}_\mathcal{S}$. Additionally, (\ref{eq: cutset of subgraph}) describes the cutset $\delta(\mathcal{S})$, i.e., the set of edges that have exactly one end in ${\tilde{\mathcal{B}}}_\mathcal{S}$.}
\textcolor{myblue}{Now, we will provide an equivalent reformulation for (\ref{eq: General Cut}) using a spanning tree characterization. This reformulation will help us verify the frequency constraints in each connected component of the grid. It also provides new insights into the way we interpret the frequency constraints.}
We will focus on the inequality (\ref{eq: General Cut: nadir min}); (\ref{eq: General Cut: nadir max})-(\ref{eq: General Cut: steady state max}) can be similarly analyzed.
\begin{prop} {Inequality (\ref{eq: General Cut: nadir min}) is equivalent to (\ref{eq: Spelled General Cut}), that is,
\begin{subequations} \label{eq: Spelled Cut}
\begin{align}
& \notag \Delta \omega_N^{\min} \le \alpha_{\s} \sum \limits_{m \in \BBB _{\s}  } { \Big( { - \Delta P_m^0 + \sum \limits_{i \in \BB _m } {\left( {1 - {x_{mi}}} \right)p_{mi}^D} } \Big) } \\
& \notag + \min \left\{ {0:(\ref{eq: IP Spanning Tree 1})-(\ref{eq: IP Spanning Tree 4})} \right\} + \sum \limits_{\left( {m,k} \right) \in \delta (\s)} {\left( {{Z_{mk}}} \right)} M_N, \\
& \qquad \qquad \qquad \qquad \forall \s \subseteq \NNN, \s  \ne \emptyset \label{eq: Spelled General Cut}
\end{align}
\end{subequations}
\vspace{-5pt}
where
\vspace{-2pt}
\begin{subequations} 
\begin{align}
& u_{mk} \le Z_{mk},\: \forall (m,k) \in \LLL(\s), \label{eq: IP Spanning Tree 1} \\
& \sum \limits_{(m,k) \in \LLL(\s)  } u_{mk}  = |\BBB _{\s} | - 1, \label{eq: IP Spanning Tree 2} \\
& \sum \limits_{\left( {m,k} \right) \in \delta (\s)} u_{mk}  \ge 1,\: \forall \s \subseteq \NNN, \s  \ne \emptyset, \NNN, \label{eq: IP Spanning Tree 3} \\
& {u_{mk}} \in \left\{ {0,1} \right\},\: \forall (m,k) \in \LLL(\s). \label{eq: IP Spanning Tree 4}
\end{align}
\end{subequations}
}
\end{prop}
\begin{proof}
The minimization problem embedded in (\ref{eq: Spelled General Cut}) has an optimal value equal to $0$ if there exists an spanning tree in $\s$. Otherwise, the problem is infeasible and the objective value will be equal to $+ \infty $, making (\ref{eq: Spelled General Cut}) redundant. \textcolor{myblue}{Here, we use the definition of a tree as a connected graph containing $n-1$ edges ($n$ is the number of nodes in the graph).} Accordingly, (\ref{eq: IP Spanning Tree 1}) ensures that the spanning tree is a subgraph of $\s$. Additionally, (\ref{eq: IP Spanning Tree 2}) and (\ref{eq: IP Spanning Tree 3}) guarantee that the spanning tree has $|\BBB _{\s} | - 1$ edges and satisfies the connectivity requirement, respectively. Finally, the last term in (\ref{eq: Spelled General Cut}) ensures that $\s$ is a component.
\end{proof}
Note that both (\ref{eq: General Cut}) and their reformulation in the form of (\ref{eq: Spelled General Cut}) have an exponential number of constraints. We will propose a cutting-plane approach to deal with this issue in Section \ref{Sec: Solution Methodology}. 
%
%\vspace{-20pt}
\subsection{\textcolor{myblue}{Overall MINLP Formulation}}
% SOCP auxiliary variables ($C_{mij}$, $ S_{mij}$, $\tilde{C}_{mk}$, $\tilde{S}_{mk}$); 
\textcolor{myblue}{
Before passing to solution methodology of the problem, let us review the overall MINLP formulation of the multi-\M ~resilient operation problem. The decision variables of this formulation are: i) the status of loads ($x_{mi}$); ii) the status of linking lines ($Z_{mk}$); iii) active and reactive flow of lines ($f_{mij}^P$, $f_{mij}^Q$, $\tilde{f}_{mk}^P$, $\tilde{f}_{mk}^{Q}$); iv) active and reactive power of DERs and loads ($p_{mg}^G$, $q_{mg}^G$, $p_{mi}^D$, $q_{mi}^D$); v) voltage magnitudes and angles ($V_{mi}$, $\theta_{mi}$); vi) active and reactive power exchange between the \M s and the linking grid ($\Delta {P_m}$, $\Delta {Q_m}$); and vii) spanning tree variable ($u_{mk}$).}
\textcolor{myblue}{
For the sake of brevity, let $\mathcal{X}$ be the set of constraints (\ref{eq: active power balance})-(\ref{eq: binary variables}) and let $\mathcal{F}$ represent the set of constraints in (\ref{eq: General Cut}). Now, we introduce $\mathcal{MINLP}(\mathcal{X},\mathcal{F})$ as follows:}
\begin{equation} \label{eq: overall MINLP Formulation}
\begin{aligned}
\textcolor{myblue}{ \vartheta = \:} & \textcolor{myblue}{\min \sum \limits_{m \in \MMM} \sum \limits_{i \in \BB _m } {\lambda _{mi}^{VOLL} \left( {1 - x_{mi}} \right) \bar p_{mi}^D} }\\
& \notag \textrm{ \textcolor{myblue}{s.t. $\: \:$ (\ref{eq: active power balance})-(\ref{eq: General Cut}).}}
\end{aligned}
\end{equation}

\vspace{-7pt}
\section{Solution Methodology} \label{Sec: Solution Methodology}
The formulation $\mathcal{MINLP}(\mathcal{X},\mathcal{F})$ is a nonconvex nonlinear optimization problem. Moreover, the developed frequency limitations in (\ref{eq: General Cut}) as well as their equivalent reformulations in (\ref{eq: Spelled General Cut}) induce exponentially many constraints. In this section, we will address these challenges.

\vspace{-7pt}
\subsection{MISOCP Reformulation and Convexification}
Observe that all the nonlinearity and nonconvexity of $\mathcal{MINLP}(\mathcal{X},\mathcal{F})$ stem from three sources:
i) the nonlinear terms $V_{mi}^2$, $V_{mi} V_{mj} \cos \left( \theta_{mi} - \theta_{mj} \right)$, and $V_{mi} V_{mj} \sin \left( \theta_{mi} - \theta_{mj} \right)$ in constraints (\ref{eq: internal active power flow})-(\ref{eq: internal reactive power flow}) and also the similar terms in (\ref{eq: external active power flow up})-(\ref{eq: external reactive power flow down}),
ii) the quadratic term $V_{mi}^2$ in constraints (\ref{eq: active load model})-(\ref{eq: reactive load model}),  
iii) the bilinear terms $x_{mi} p_{mi}^D$ and $x_{mi} q_{mi}^D$ in constraints (\ref{eq: active power balance})-(\ref{eq: boundary reactive power balance}) and (\ref{eq: General Cut}).
In this section, we will convexify/linearize the aforementioned terms, leading to an MISOCP relaxation of the \MM ~resilient operation problem.
\subsubsection{MISOCP Relaxation of AC Power Flow Equations} \label{sec: MISOCP Relaxation of AC Power Flow Equations}

Based on the recent development in SOCP relaxation of standard AC-OPF \cite{kocuk-2016-StrongSOCP}, we define the following auxiliary variables for each $\left( {i,j} \right) \in  \LL_m$ and $m \in \MMM$:
% \vspace{-15pt}
\begin{subequations} \label{eq: ViVjcos and ViVjsin}
\begin{align}
& C_{mij} := V_{mi} V_{mj} \cos \left( \theta_{mi} - \theta_{mj} \right), \\
& S_{mij} := V_{mi} V_{mj} \sin \left( \theta_{mi} - \theta_{mj} \right).
\end{align}
\end{subequations}
Observe that (\ref{eq: ViVjcos and ViVjsin}) implies (\ref{eq: S and C declaration}), that is 
\begin{subequations} \label{eq: S and C declaration}
\begin{align}
& C_{mij}^2 + S_{mij}^2 = C_{mii} C_{mjj}, \label{eq: S2 + C2 = cc} \\
& S_{mij} = - S_{mji}, \: C_{mij} = C_{mji}.\label{eq: S = -S  C = C} 
\end{align}
\end{subequations}
Similarly, we define $\tilde{C}_{mk}  := \tilde{V}_m \tilde{V}_k \cos( \theta_m - \theta_k )$ and $\tilde{S}_{mk}  = \tilde{V}_m \tilde{V}_k \sin ( \theta_m - \theta_k )$ for each $(m,k) \in \LLL$, and the following constraints will be inferred:
\begin{subequations} \label{eq: tide S and C declaration}
\begin{align}
& \tilde{C}_{mk}^2 + \tilde{S}_{mk}^2 = \tilde{C}_{mm} \tilde{C}_{kk},  \label{eq: tilde S2 + C2 = cc}  \\
& \tilde{S}_{mk} = - \tilde{S}_{km}, \: \tilde{C}_{mk} = \tilde{C}_{km}.\label{eq: tilde S = -S  C = C}
\end{align}
\end{subequations}
\textcolor{myblue}{Note that the convex relaxation of (\ref{eq: S2 + C2 = cc}) and (\ref{eq: tilde S2 + C2 = cc}) are: 
% (\ref{eq: Relaxed S2 + C2 = cc}) and (\ref{eq: Relaxed tilde S2 + C2 = cc}) are SOCP constraints associated with the unrelaxed counterparts .
}
\begin{subequations} \label{eq: SOCP relaxation counterparts of S2 + C2 = cc}
\begin{align}
& C_{mij}^2 + S_{mij}^2 \le C_{mii} C_{mjj}, \label{eq: Relaxed S2 + C2 = cc} \\
& \tilde{C}_{mk}^2 + \tilde{S}_{mk}^2 \le \tilde{C}_{mm} \tilde{C}_{kk}. \label{eq: Relaxed tilde S2 + C2 = cc} 
\end{align}
\end{subequations}
%
%\end{subequations}
With a change of variables for each $m \in \MMM$ and $\left( {i,j} \right) \in  \LL_m$, constraints (\ref{eq: internal active power flow}) and (\ref{eq: internal reactive power flow}) can be written as
\begin{subequations} \label{eq: relaxation counterparts of internal power flow}
\begin{align}
& f_{mij}^P =  G_{mij} \Big( C_{mii} - C_{mij} \Big) - B_{mij} S_{mij}, \label{eq: SOCP active power flow} \\
& f_{mij}^Q = - B_{mij} \Big( C_{mii} - C_{mij} \Big) - G_{mij} S_{mij},
\end{align}
\end{subequations}
and the voltage bound (\ref{eq: voltage limit}) for each $m \in \MMM$ and $i \in (\BB_m \cup \BBBB_m)$ is transformed into 
\begin{equation} \label{eq: SOCP internal voltage bounds}
(V_{mi}^{\min})^2 \le C_{mii} \le (V_{mi}^{\max})^2 . 
\end{equation}
Likewise, a change of variables for each $(m,k) \in \LLL$ leads to the constraints (\ref{eq: relaxation counterparts of external power flow}) as the counterparts of (\ref{eq: external active power flow up})-(\ref{eq: external reactive power flow down}):
\begin{subequations} \label{eq: relaxation counterparts of external power flow}
\begin{align}
& - \tilde{f}_{mk}^{\textcolor{myblue}{P}} + \tilde{G}_{mk} \Big( \tilde{C}_{mm} - \tilde{C}_{mk} \Big) - \tilde{B}_{mk} \tilde{S}_{mk} \label{eq: SOCP external active power flow} \\
& \notag \quad \qquad + \left( 1 - Z_{mk} \right) M_{mk} \ge 0, \\
& - \tilde{f}_{mk}^P +  \tilde{G}_{mk} \Big( \tilde{C}_{mm} - \tilde{C}_{mk} \Big) - \tilde{B}_{mk} \tilde{S}_{mk} \\
& \notag \quad \qquad  - \left( 1 - Z_{mk} \right) M_{mk} \le 0, \\
& - \tilde{f}_{mk}^Q -  \tilde{B}_{mk} \Big( \tilde{C}_{mm} - \tilde{C}_{mk} \Big) - \tilde{G}_{mk} \tilde{S}_{mk} \\
& \notag \quad \qquad + \left( 1 - Z_{mk} \right) M'_{mk} \ge 0, \\
& - \tilde{f}_{mk}^Q -  \tilde{B}_{mk} \Big( \tilde{C}_{mm} - \tilde{C}_{mk} \Big) - \tilde{G}_{mk} \tilde{S}_{mk} \\
& \notag \quad \qquad - \left( {1 - {Z_{m,k} }} \right) M'_{mk} \le 0 , 
\end{align}
\end{subequations}
and similarly the voltage bound (\ref{eq: external voltage limit}) for each $ m \in \MMM$ and $i \in \BBB$ can be written as: 
\begin{equation} \label{eq: SOCP external voltage bounds}
(\tilde{V}_{m}^{\min})^2 \le \tilde{C}_{mm} \le (\tilde{V}_{m}^{\max})^2 . 
\end{equation}
\subsubsection{MISOCP Relaxation of ZIP Load Models} \label{subsubsec: MISOCP Relaxation of ZIP Load Model}
\textcolor{myblue}{Using the SOCP auxiliary variables defined in Section \ref{sec: MISOCP Relaxation of AC Power Flow Equations}, the ZIP load models (\ref{eq: active load model}) and (\ref{eq: reactive load model}) can be written as (\ref{eq: SOCP replaced load model}) for each $m \in \MMM$ and $i \in (\BB_m \cup \BBBB_m)$, that is}
\begin{subequations} \label{eq: SOCP replaced load model}
\begin{align}
& p_{mi}^D = \overline p_{mi}^D \Big( \kappa_{mi}^{PI} C_{mii} + \kappa_{mi}^{PC} \sqrt{C_{mii}} + \kappa_{mi}^{PP} \Big), \label{eq: new variable active load model} \\
%& \notag \qquad \qquad \forall i \in (\BB_m \cup \BBBB_m), \: m \in \MMM \\
%
& q_{mi}^D = \overline q_{mi}^D \Big( \kappa_{mi}^{QI} C_{mii} + \kappa_{mi}^{QC} \sqrt{C_{mii}} + \kappa_{mi}^{QP} \Big). \label{eq: new variable reactive load model} 
%& \notag \qquad \qquad \forall i \in (\BB_m \cup \BBBB_m), \: m \in \MMM \\
\end{align}
\end{subequations}
% The next step is, therefore, to convexify the constraints (\ref{eq: new variable active load model}) and (\ref{eq: new variable reactive load model}). 
The convex relaxation of these two constraints
%for each $i \in (\BB_m \cup \BBBB_m)$ and $m \in \MMM$ 
can be written as
\begin{subequations} \label{eq: ZIP relaxation}
\begin{align}
& \overline p_{mi}^D \Big( \kappa_{mi}^{PI} C_{mii} + \kappa_{mi}^{PC} \sqrt{C_{mii}} + \kappa_{mi}^{PP} \Big) \label{eq: relaxed active load model} - p_{mi}^D \ge 0, \\
& \overline q_{mi}^D \Big( \kappa_{mi}^{QI} C_{mii} + \kappa_{mi}^{QC} \sqrt{C_{mii}} + \kappa_{mi}^{QP} \Big) \label{eq: relaxed reactive load model} - q_{mi}^D \ge 0.
\end{align}
\end{subequations}
Since the variable $C_{mii}$ is bounded by the closed interval $[C_{mii}^{\min},C_{mii}^{\max}]$, the convex relaxation  (\ref{eq: ZIP relaxation}) can be tighten by introducing the following two hyperplanes which pass through the end points for each $i \in (\BB_m \cup \BBBB_m)$ and $m \in \MMM$:
\begin{subequations} \label{eq: ZIP relaxation hyperplanes}
\begin{align}
& p_{mi}^D - p_{mi}^{D,\min}  \ge \frac{p_{mi}^{D,\max} - p_{mi}^{D,\min}}{C_{mii}^{\max} - C_{mii}^{\min}} \left( C_{mii} - C_{mii}^{\min} \right), \\ 
&   q_{mi}^D - q_{mi}^{D,\min}  \ge \frac{q_{mi}^{D,\max} - q_{mi}^{D,\min}}{C_{mii}^{\max} - C_{mii}^{\min}} \left( C_{mii} - C_{mii}^{\min} \right).
\end{align}
\end{subequations}
\begin{prop} {Constraints (\ref{eq: relaxed active load model}) and (\ref{eq: relaxed reactive load model}) are SOCP representable in terms of $C_{mii}^2$.}
\end{prop}
\begin{proof}
We focus on constraint (\ref{eq: relaxed active load model}); constraint (\ref{eq: relaxed reactive load model}) is similarly analyzed. First, we rearrange and square both sides of the constraint for each $i \in (\BB_m \cup \BBBB_m)$  and $m \in \MMM$ such that
\begin{equation} \label{eq: SOCP representation active load model}
\begin{aligned}
& \kappa_{mi}^{PC} \sqrt{C_{mii}} \ge \frac{p_{mi}^D}{\overline p_{mi}^D} - \kappa_{mi}^{PI} C_{mii} - \kappa_{mi}^{PP}
\end{aligned}
\end{equation}
\begin{equation} \label{eq: squared SOCP representation active load model}
\begin{aligned}
& \left( \kappa_{mi}^{PC} \right)^2 C_{mii} \ge \left( \frac{p_{mi}^D}{\overline p_{mi}^D} - \kappa_{mi}^{PI} C_{mii} - \kappa_{mi}^{PP} \right)^2.
\end{aligned}
\end{equation}

Note that $ C_{mii} = (\frac{C_{mii}+1}{2})^2 - (\frac{C_{mii}-1}{2})^2 $, therefore (\ref{eq: squared SOCP representation active load model}) can be written as the following SOCP constraint for each $i \in (\BB_m \cup \BBBB_m)$ and $m \in \MMM$:
\begin{equation} \label{eq: final squared SOCP representation active load model}
\begin{aligned}
\left( \kappa_{mi}^{PC} \right)^2 &  \left( \frac{C_{mii}+1}{2}\right)^2 \ge  \left( \kappa_{mi}^{PC} \right)^2  \left(\frac{C_{mii}-1}{2}\right)^2 \\
& + \left( \frac{p_{mi}^D}{\overline p_{mi}^D} - \kappa_{mi}^{PI} C_{mii} - \kappa_{mi}^{PP} \right)^2.
\end{aligned}
\end{equation}
\end{proof}
\vspace{-5pt}
\subsubsection{Linearizion of the Bilinear Terms} \label{subsubsec: Linearizion of the Bilinear Terms}
Finally, let us linearize the bilinear terms $x_{mi} p_{mi}^D$ and $x_{mi} q_{mi}^D$ in (\ref{eq: active power balance})-(\ref{eq: boundary reactive power balance}) and (\ref{eq: General Cut}), where each bilinear term involves the product of a binary variable and a nonnegative continuous variable. We linearize these disjunctive terms via the big-M method by introducing auxiliary semi-continuous variables $\rho_{mi} := x_{mi} p_{mi}^D$ and $\sigma_{mi} := x_{mi} q_{mi}^D$ and defining additional constraints for each $i \in (\BB_m \cup \BBBB_m)$  and $m \in \MMM$:
\begin{subequations} \label{eq: Bilinear Linearizion}
\begin{align}
& -\left( 1 - x_{mi} \right) \check{M}_{mi}^p  \le \rho_{mi} - p_{mi}^D \le \check{M}_{mi}^p \left( 1 - x_{mi} \right), \\
& - x_{mi} \check{M}_{mi}^p \le \rho_{mi} \le \check{M}_{mi}^p x_{mi}, \\
& - \left( 1 - x_{mi} \right) \check{M}_{mi}^q \le \sigma_{mi} - q_{mi}^D \le \check{M}_{mi}^q \left( 1 - x_{mi} \right), \\
& - x_{mi} \check{M}_{mi}^q \le \sigma_{mi} \le \check{M}_{mi}^q x_{mi}.
\end{align}
\end{subequations}

In order to reduce the integrality gap in (\ref{eq: Bilinear Linearizion}), the big-Ms (i.e., $\check{M}_{mi}^p$ and $\check{M}_{mi}^q$) should be as small as possible, and it is usually challenging to determine correct values for them to use for each specific implementation. However, in this particular application, we can set $\check{M}_{mi}^p=\overline p_{mi}^D$ and $\check{M}_{mi}^q=\overline q_{mi}^D$. Note that these data (i.e., the upper bounds of active and reactive loads) are usually available in any system.
Now, substituting the auxiliary variables $\rho_{mi}$ and $\sigma_{mi}$ into the constraints (\ref{eq: active power balance})-(\ref{eq: boundary reactive power balance}), we get the linear constraints (\ref{eq: linearized active power balance})-(\ref{eq: linearized reactive power balance}) for each $m \in \MMM$, $i \in \BB_m$, and also the constraints (\ref{eq: linearized boundary active power balance})-(\ref{eq: linearized boundary reactive power balance}) for each $m \in \MMM$, $i \in \BBBB_m$:
\begin{subequations} \label{eq: Substitute bilinear terms}
\begin{align}
& \sum \limits_{g:(g,i) \in {\OO}_m} p_{mg}^G - \rho_{mi} = \sum \limits_{\left( {i,j} \right) \in \LL_m }  f_{mij}^P, \label{eq: linearized active power balance} \\
& \sum\limits_{g:(g,i) \in  {\OO}_m} q_{mg}^G - \sigma_{mi} = \sum \limits_{\left( {i,j} \right) \in \LL_m } f_{mij}^Q, \label{eq: linearized reactive power balance} \\
& \sum \limits_{g:(g,i) \in {\OO}_m} {p_{mg}^G}  - \rho_{mi} + \Delta {P_m} = \sum \limits_{\left( {i,j} \right) \in \LL_m } f_{mij}^P, \label{eq: linearized boundary active power balance} \\
& \sum \limits_{g:(g,i) \in {\OO}_m} {q_{mg}^G}  - \sigma_{mi} + \Delta {Q_m}= {\rm{ }}\sum\limits_{\left( {i,j} \right) \in \LL_m } f_{mij}^Q. \label{eq: linearized boundary reactive power balance}
\end{align}
\end{subequations}
\textcolor{myblue}{
Complementarily, the frequency constraints (\ref{eq: General Cut}) can be written as (\ref{eq: Final Cut }) for each $\s \subseteq \NNN, \s  \ne \emptyset$, where the indicator function $\mathcal{I}_M$ is modeled using the big-M method and the bilinear terms are replaced with their linear counterparts: }
\begin{subequations} \label{eq: Final Cut }
\begin{align}
& \notag \Delta \omega_N^{\min} \le \alpha_{\s} \sum \limits_{m \in \BBB _{\s}  } { \Big( { - \Delta P_m^0 + \sum \limits_{i \in \BB _m } {\left( {p_{mi}^D - \rho_{mi} } \right)} } \Big) } \\
& + \sum \limits_{ (m,k) \in \LLL(\s) } {\left( 1 - Z_{mk} \right)} M_N + \sum \limits_{\left( {m,k} \right) \in \delta (\s)} {\left( {{Z_{mk}}} \right)} M_N, \label{eq: Final Cut: nadir min} \\
& \notag \textcolor{myblue}{   \Delta \omega_N^{\max} \ge \alpha_{\s} \sum \limits_{m \in \BBB _{\s}  } { \Big( { - \Delta P_m^0 + \sum \limits_{i \in \BB _m } {\left( {p_{mi}^D - \rho_{mi} } \right) } } \Big) }   }\\
& \textcolor{myblue}{  - \sum \limits_{ (m,k) \in \LLL(\s) } {\left( 1 - Z_{mk} \right)} M_N  - \sum \limits_{\left( {m,k} \right) \in \delta (\s)} {\left( {{Z_{mk}}} \right)} M_N,  }\label{eq: Final Cut: nadir max} \\
& \notag \Delta \omega_{ss}^{\min} \le \beta_{\s} \sum \limits_{m \in \BBB _{\s}  } { \Big( { - \Delta P_m^0 + \sum \limits_{i \in \BB _m } {\left( {p_{mi}^D - \rho_{mi} } \right) } } \Big) } \\
& + \sum \limits_{ (m,k) \in \LLL(\s) } {\left( 1 - Z_{mk} \right)} M_{ss}  + \sum \limits_{\left( {m,k} \right) \in \delta (\s)} {\left( {{Z_{mk}}} \right)} M_{ss}, \label{eq: Final Cut: steady state min} \\
 & \notag \textcolor{myblue}{   \Delta \omega_{ss}^{\max} \ge \beta_{\s} \sum \limits_{m \in \BBB _{\s}  } { \Big( { - \Delta P_m^0 + \sum \limits_{i \in \BB _m } {\left( {p_{mi}^D - \rho_{mi} } \right)  } } \Big) }  } \\
 & \textcolor{myblue}{   - \sum \limits_{ (m,k) \in \LLL(\s) } {\left( 1 - Z_{mk} \right)} M_{ss} - \sum \limits_{(m,k) \in \delta (\s)} {\left( Z_{mk} \right)} M_{ss}. } \label{eq: Final Cut: steady state max}
\end{align}
\end{subequations}
\subsubsection{\textcolor{myblue}{Overall MISOCP Formulation}}
\textcolor{myblue}{Before proceeding further with the analysis, let us define the set $\mathcal{R}$ as the set of constraints (\ref{eq: internal active flow limit loss}), (\ref{eq: external active power balance}), (\ref{eq: external reactive power balance}), (\ref{eq: external active flow limit})-(\ref{eq: external active flow limit loss}), (\ref{eq: generator ramp limits})-(\ref{eq: binary variables}), (\ref{eq: S = -S  C = C}), (\ref{eq: tilde S = -S  C = C}), (\ref{eq: SOCP relaxation counterparts of S2 + C2 = cc})-(\ref{eq: SOCP external voltage bounds}), (\ref{eq: ZIP relaxation}), (\ref{eq: ZIP relaxation hyperplanes}), (\ref{eq: Bilinear Linearizion}), and (\ref{eq: Substitute bilinear terms}). Recall that $\mathcal{F}$ is the set of frequency constraints. Now, we can formally define $\mathcal{MISOCP}(\mathcal{R},\mathcal{F})$, as the MISOCP relaxation of the multi-\M ~resilient operation problem:}
%, where $\mathcal{R}$ is the set of constraints in (\ref{eq: Final self healing MISOCP relaxation}):
\begin{equation} \label{eq: Final self healing MISOCP relaxation}
\begin{aligned}
\textcolor{myblue}{ \psi} \textcolor{myblue}{=} & \textcolor{myblue}{\min \sum \limits_{m \in \MMM} \sum \limits_{i \in \BB _m } {\lambda _{mi}^{VOLL} \left( {1 - x_{mi}} \right) \bar p_{mi}^D} }\\
& \notag \textrm{ \textcolor{myblue}{ s.t. $\:$ (\ref{eq: internal active flow limit loss}), (\ref{eq: external active power balance}), (\ref{eq: external reactive power balance}), (\ref{eq: external active flow limit})-(\ref{eq: external active flow limit loss}), (\ref{eq: generator ramp limits})-(\ref{eq: binary variables}), (\ref{eq: S = -S  C = C}),} }\\
& \notag \qquad \: \textrm{ \textcolor{myblue}{ (\ref{eq: tilde S = -S  C = C}), (\ref{eq: SOCP relaxation counterparts of S2 + C2 = cc})-(\ref{eq: SOCP external voltage bounds}), (\ref{eq: ZIP relaxation}), (\ref{eq: ZIP relaxation hyperplanes}), (\ref{eq: Bilinear Linearizion})-(\ref{eq: Final Cut }).}}
%
% & \sum \limits_{g:(g,i) \in {\OO}_m} p_{mg}^G - \rho_{mi} = \sum \limits_{\left( {i,j} \right) \in \LL_m }  f_{mij}^P, \forall i \in \BB_m, m \in \MMM \label{eq: linearized active power balance} \\
% & \sum\limits_{g:(g,i) \in  {\OO}_m} q_{mg}^G - \sigma_{mi} = \sum \limits_{\left( {i,j} \right) \in \LL_m } f_{mij}^Q, \forall i \in \BB_m, m \in \MMM \label{eq: linearized reactive power balance} \\
% & \notag  \sum \limits_{g:(g,i) \in {\OO}_m} {p_{mg}^G}  - \rho_{mi} + \Delta {P_m} = \sum \limits_{\left( {i,j} \right) \in \LL_m } f_{mij}^P, \\
% & \qquad \qquad \forall i \in \BBBB_m, m \in \MMM \label{eq: linearized boundary active power balance} \\
% & \notag  \sum \limits_{g:(g,i) \in {\OO}_m} {q_{mg}^G}  - \sigma_{mi} + \Delta {Q_m}= {\rm{ }}\sum\limits_{\left( {i,j} \right) \in \LL_m } f_{mij}^Q, \\
% & \qquad \qquad \forall i \in \BBBB_m, m \in \MMM. \label{eq: linearized boundary reactive power balance}
\end{aligned}
\end{equation}

\textcolor{myblue}{It remains to deal with the exponential number of constraints in $\mathcal{F}$. This is the topic of the next section.}
% define the set $\mathcal{A}$ as the feasible region of the problem (\ref{eq: self-healing relaxation}) where the non-convex ZIP load model and the bilinear terms in the constraints are replaced respectively by the MISOPC relaxation of the load model (developed in Section \ref{subsubsec: MISOCP Relaxation of ZIP Load Model}) and the linearized terms (developed in Section \ref{subsubsec: Linearizion of the Bilinear Terms}).  
%
\vspace{-5pt}
\subsection{Cutting Plane Algorithm for Frequency Constraints}
In this section, we propose a cutting plane approach to solve $\mathcal{MISOCP}(\mathcal{R},\mathcal{F})$. \textcolor{myblue}{The idea is to construct $\{ \mathcal{F}_k \}_{k \ge 0}$, that is a sequence of relaxations of the set $\mathcal{F}$, and dynamically update $\mathcal{F}_k$ to obtain stronger relaxations in each iteration. Recall that the set $\mathcal{F}$ contains exponentially many frequency constraints.}

With this aim in mind, let $\mathcal{C}_{\s}^1$, $\mathcal{C}_{\s}^2$, $\mathcal{C}_{\s}^3$, and $\mathcal{C}_{\s}^4$ denote respectively the constraints (\ref{eq: Final Cut: nadir min}), (\ref{eq: Final Cut: nadir max}), (\ref{eq: Final Cut: steady state min}), and (\ref{eq: Final Cut: steady state max}), for a given connected component $\s$ of the linking grid, where $\s \subseteq \NNN, \s  \ne \emptyset$. Moreover, let the graph $\tilde{\mathcal{N}}^*$ represent the configuration of the linking grid for a given solution to $\mathcal{MISOCP}(\mathcal{R},\mathcal{F}_k)$, and let  $\mathcal{Q} = \{ \s _{\upsilon_1}, \s _{\upsilon_2}, ..., \s _{\upsilon_N} \}$ denote the set of connected components of $\tilde{\mathcal{N}}^*$ where $ \{ {\upsilon_1}, {\upsilon_2}, ..., {\upsilon_N} \} \subseteq \{1,2,..., |\BBB|\}$. For each component in $\mathcal{Q}$, we check the inequalities $\{ \mathcal{C}_{\s}^{\gamma} {\}}_{\gamma=1}^{\textcolor{myblue}{4}} $; if any frequency violation is detected, the corresponding valid inequality will be added to the set $\mathcal{F}_k$.
% Note that $\{ \mathcal{C}_{\s}^{\gamma} {\}}_{\gamma=1}^{\textcolor{myblue}{4}} $ are valid inequalities for $\mathcal{A}$. 
\textcolor{myblue}{In other words, let $\mathcal{A}$ be the set of feasible solutions to the problem $\mathcal{MISOCP}(\mathcal{R},\mathcal{F})$. In each iteration, if an optimal solution of $\mathcal{MISOCP}(\mathcal{R},\mathcal{F}_k)$ is in the set $\mathcal{A}$, we stop since we have already found an optimal solution to $\mathcal{MISOCP}(\mathcal{R},\mathcal{F})$. Otherwise, we generate a cut and add it to $\mathcal{F}_{k}$ to separate the point from the set $\mathcal{A}$ and obtain stronger relaxations in the next iteration. Algorithm \ref{Alg: Cutting generation algorithm} provides the details of the proposed cutting plane approach.}  
% We start with the trivial relaxation $\mathcal{F}_0 = \emptyset$, and at each iteration if an optimal solution of $\mathcal{MISOCP}(\mathcal{R},\mathcal{F}_k)$ is in the set $\mathcal{A}$, we stop since we have already found an optimal solution to $\mathcal{MISOCP}(\mathcal{R},\mathcal{F})$. Otherwise, we generate a cut and add it to $\mathcal{F}_{k+1}$ to separate the point from the set $\mathcal{A}$ and obtain stronger relaxations in the next iteration.

% In this way, stronger lower bounds are obtained in each iteration until converging to the frequency-constrained optimal solution. 

% In order to separate a given solution of $\mathcal{MISOCP}(\mathcal{R},\mathcal{F})$ from $\mathcal{A}$, we propose  in . 

% In order to separate a given solution of $\mathcal{MISOCP}(\mathcal{R},\mathcal{F})$ from $\mathcal{A}$, we propose a cutting plane approach in Algorithm \ref{Alg: Cutting generation algorithm}. 

%

% : Solve the linear programming relaxation obtained by ignoring
% the integrality requirements on x; if this relaxation is unbounded or infeasible,
% then stop: the mixed integer linear progam is infeasible or unbounded; otherwise
% let (x, y) be an optimal extreme point solution; if (x, y) ∈ S, then stop: (x, y)
% is an optimal solution of the mixed integer linear program; otherwise generate a
% cut with respect to (x, y), add it to the previous linear programming relaxation,
% solve this improved relaxation, and repeat. I

\begin{algorithm} [t]
 \caption{Multi-\M ~resilient operation algorithm}
 \label {Alg: Cutting generation algorithm}
 \begin{algorithmic}[1]
 \STATE Initialize $k \gets 0$, $\mathcal{F}_k \gets  \emptyset$, Flag $\gets$ \texttt{NO} 
 \WHILE{Flag $=$ \texttt{NO} }
 \STATE \textcolor{myblue}{ Solve $\mathcal{MISOCP}(\mathcal{R},\mathcal{F}_k)$ to obtain the graph $\tilde{\mathcal{N}}^*$ representing the optimal configuration of the linking grid}
% \STATE $LB_k \gets {{{\psi^*}}} $ 
%  \STATE $\tilde{\mathcal{N}}^*$ $\gets$ The graph representing the configuration of the linking grid obtained from the optimal solution of $\mathcal{MISOCP}(\mathcal{R},\mathcal{F})$
%\STATE $\mathcal{Q}$ $\gets$ \textsc{ConnectedComponent}($\tilde{\mathcal{N}}^*$) 

\STATE \textcolor{myblue}{Compute $\mathcal{Q} = \{ \s _{\upsilon_1}, \s _{\upsilon_2}, ..., \s _{\upsilon_N} \}$ as the set of connected components of $\tilde{\mathcal{N}}^*$} 
\STATE Flag $\gets$ \texttt{YES}
 \FOR {$ \upsilon = {\upsilon_1}$ \textbf{to} ${\upsilon_N}$}
     \FOR {$\gamma =1$ \textbf{to} ${\textcolor{myblue}{4}}$}
        \IF {$ \s _{\upsilon}$ violates $\mathcal{C}_{\s _ \upsilon}^{\gamma} $}
          \STATE Flag $\gets$ \texttt{NO}
          \STATE $\mathcal{F}_k \gets \mathcal{F}_k \cup \{ \mathcal{C}_{\s _ \upsilon}^{\gamma} \}$
        \ENDIF
     \ENDFOR
 \ENDFOR
 \STATE $k \gets k+1$
  \ENDWHILE
 \end{algorithmic} 
 \end{algorithm}
As can be seen, in Algorithm \ref{Alg: Cutting generation algorithm}, we need a function to return the connected components of the undirected graph $\tilde{\mathcal{N}}^*$. Recall that a connected component of an undirected graph is a maximal connected subgraph of the graph. This function can be implemented via depth-first or breadth-first algorithm. See \cite{West-2001-graph-theory} for details.
%
% In Algorithm \ref{Alg: Cutting generation algorithm}, the function \textsc{ConnectedComponent}($\tilde{\mathcal{N}}^*$) returns the connected components of the undirected graph $\tilde{\mathcal{N}}^*$. Recall that a connected component of an undirected graph is a maximal connected subgraph of the graph. The function \textsc{ConnectedComponent}($\tilde{\mathcal{N}}^*$) can be implemented via depth-first or breadth-first algorithm. See \cite{West-2001-graph-theory} for details.
%
\begin{theorem} Algorithm \ref{Alg: Cutting generation algorithm} converges to an optimal solution of the \textcolor{myblue}{MISOCP-based} \MM ~resilient operation problem, \textcolor{myblue}{i.e., $\mathcal{MISOCP}(\mathcal{R},\mathcal{F})$,} in a finite number of iterations.
\end{theorem}
\begin{proof}
Let $x_{mi}^*$ and $Z_{mk}^*$ be an optimal solution to the problem $\mathcal{MISOCP}(\mathcal{R},\mathcal{F}_0)$ where $i \in (\BB_m \cup \BBBB_m), \: m \in \MMM$, $(m,k) \in \LLL$, and $\mathcal{F}_0 = \emptyset$. If $x_{mi}^*$ and $Z_{mk}^*$ satisfy (\ref{eq: Final Cut }), then Algorithm \ref{Alg: Cutting generation algorithm} converges to the optimal solution in one iteration. Otherwise, in each iteration, at least one constraint will be added to the set $\mathcal{F}_k$. We observe that the total number of constraints in (\ref{eq: Final Cut }) is ${\textcolor{myblue}{4}} r$, where $r$ is the number of possible connected components of $\NNN$. Since each connected component is examined at most once in this algorithm, the number of iterations needed for the convergence of the algorithm is less than ${\textcolor{myblue}{4}} r$.
\end{proof}
\section{Computational Experiments} \label{Sec: Computational Experiments}
%\subsection{Test Systems and Simulation Configuration}
In this section, the performance of the proposed framework for the multi-\M ~resilient operation problem is thoroughly evaluated. All simulations are conducted on a $64$-bit PC with Intel Core i$7$ CPU $2.8$ GHz processor and $16$ GB RAM. The algorithm is implemented in the GAMS IDE environment \cite{GAMS}. We use BONMIN V$1.8$ \cite{BONMIN} to solve MINLPs and CPLEX V$12.4$ \cite{CPLEX} to solve the MISOCPs. Moreover, we use the $39$-bus \MM ~network (depicted in Fig. \ref{Fig: Multi Microgrid}) as our test system. This network is composed of six DERs, whose technical data are given in Table \ref{tab: Technical Data of DERs}. Feeders’ and loads' data are adopted from different portions of a standard IEEE distribution test system whose data can be found in \cite{Baran-1989-IEEE33Bus}. To have a more realistic study, five different load types (i.e., general, residential, agricultural, commercial, and industrial) with different VOLLs are taken into account (see Fig. 5 in \cite{Amin-2018-HICSS}). Finally, the \M s' dynamic data is given in Table \ref{tab: Dynamic Parameters}. 
%
% \begin{table} [htb]
% \caption{Technical Data of DERs}
% \label{tab: Technical Data of DERs}
% \hspace{-5mm}
% \begin{tabular}{c|c|c|c|c|c|c|}
% %\begin{tabular} |M{0.19cm}|M{1cm}|M{0.3cm}|M{1.2cm}|M{1.9cm}|M{1.5cm}|M{1cm}|}
% \cline{2-7}
%                                           & \multicolumn{6}{c|}{\textbf{DERs}}   \\ \hline
% \multicolumn{1}{|c|}{\textbf{Parameters}}  & G$_1$  
%                             & G$_2$
%                                     & G$_3$ 
%                                             & G$_4$  
%                                                     & G$_5$  
%                                                             & G$_6$  \\ \hline 
% \multicolumn{1}{|c|} {$p^{G,\min }$ [kW]}        & $100$      & $100$    & $100$    & $100$    & $100$    & $100$  \\ \hline
% \multicolumn{1}{|c|} {$p^{G,\max }$[kW]}         & $500$      & $200$    & $500$    & $200$    & $200$    & $500$  \\ \hline
% \multicolumn{1}{|c|} {$q^{G,\min }$[kVAr]}       & $-500$     & $-200$   & $-500$   & $-200$   & $-200$   & $-500$   \\ \hline
% \multicolumn{1}{|c|} {$q^{G,\max }$[kVAr]}       & $500$      & $200$    & $500$    & $200$    & $200$    & $500$     \\ \hline
% \multicolumn{1}{|c|} {$R^D \backslash R^D$[kW/min]}  & $200$   & $100$  & $200$    & $100$    & $100$    & $200$ \\ \hline
% \end{tabular}
% \end{table}
%%%
\vspace{-10pt}
\begin{table} [htb]
\centering
\caption{Technical Data of DERs}
\label{tab: Technical Data of DERs}
\hspace{-5mm}
\begin{tabular}{c|c|c|c|c|c|c|}
%\begin{tabular} |M{0.19cm}|M{1cm}|M{0.3cm}|M{1.2cm}|M{1.9cm}|M{1.5cm}|M{1cm}|}
\cline{2-7}
                                          & \multicolumn{6}{c|}{\textbf{DERs}}   \\ \hline
\multicolumn{1}{|c|}{\textbf{Parameters}}  & G$_1$  
                            & G$_2$
                                    & G$_3$ 
                                            & G$_4$  
                                                    & G$_5$  
                                                            & G$_6$  \\ \hline 
\multicolumn{1}{|c|} {$p^{G,\min }$ [$\times 100$ kW]}        & $1$      & $1$    & $1$    & $1$    & $1$    & $1$  \\ \hline
\multicolumn{1}{|c|} {$p^{G,\max }$[$\times 100$ kW]}         & $5$      & $2$    & $5$    & $2$    & $2$    & $5$  \\ \hline
\multicolumn{1}{|c|} {$q^{G,\min }$[$\times 100$ kVAr]}       & $-5$     & $-2$   & $-5$   & $-2$   & $-2$   & $-5$   \\ \hline
\multicolumn{1}{|c|} {$q^{G,\max }$[$\times 100$ kVAr]}       & $5$      & $2$    & $5$    & $2$    & $2$    & $5$     \\ \hline
\multicolumn{1}{|c|} {$R^D \backslash R^D$[$\times 100$ kW/min]}  & $2$   & $1$  & $2$    & $1$    & $1$    & $2$ \\ \hline
\end{tabular}
\end{table}
%
%%%
\vspace{-10pt}
\begin{table} [htb]
\centering
\caption{Dynamic Parameters of the VSC Controller in each \M}
\label{tab: Dynamic Parameters}
\begin{tabular}{|l|l||l|l||l|l|}
\hline
Parameter            & Value    & Parameter                        & Value   & Parameter      & Value   \\ \hline
$H$ [sec.]          & $0.9$    & $D$                              & $1$       & $T'$ [sec.]     & $0.1$    \\ \hline
$R$                 & $0.08$   & $\Delta \omega_N$ [Hz]    & $0.5$     & $V_{Base}$[kV]  & $12.66$ \\ \hline
$T$ [sec.]          & $0.008$  & $\Delta \omega_{ss}$ [Hz] & $0.1$     & $S_{Base}$[MW]  & $5$   \\ \hline
\end{tabular}
\end{table}

We assume that all \M s in Fig. \ref{Fig: Multi Microgrid} were initially connected to the main grid through the dashed lines (in red). Subsequent to islanding, these lines along with the main circuit breaker trip. The proposed MISOCP-based resilient operation approach determines the optimal strategy which may include re-closing the dashed lines and switching the dotted lines (in gray), leading to different configurations for the distribution network. In order to evaluate our framework, we compare it with the following two schemes:
\begin{itemize}
\item \textit{MINLP-Based Scheme:} In this scheme, we follow our resilient operation scheme; however, we use $\mathcal{MINLP}(\mathcal{X},\mathcal{F})$ as the decision support tool in Algorithm \ref{Alg: Cutting generation algorithm}.

\item \textit{Conventional UFLS Scheme:} In this scheme, subsequent to islanding of the distribution network, each \M ~individually enters the island mode where the conventional UFLS relays will curtail the necessary blocks of loads until reaching the equilibrium point. The settings of these relays are obtained from \cite{Amin-2018-HICSS}.

%\cite{Abedini-2014-ConventionalUFLS}.
\end{itemize}
%
% \begin{figure}[t]
%  \includegraphics*[width=2.75in, keepaspectratio=true]{TestCase.pdf}
%  \centering
%   \caption{Schematic diagram of the \MM ~network adopted in Case II.}
%   \label{Fig: Test Case 2}
% \end{figure}
%
%
\vspace{-7pt}
\subsection{Comparison with the MINLP-Based Scheme}
\subsubsection{\textcolor{myblue}{Solution and Computation Time}}
Table \ref{tab: Computation Time} provides a comparison between the MINLP-based and MISOCP-based schemes considering different severities for the islanding event (we define severity as the amount of power flow from the main grid to the distribution network before the islanding). The computation times in this table are obtained using a relative optimality criterion (i.e., Optcr) of zero.
% \textcolor{myblue}{
% Considering the second islanding event with severity of $3200$~kW, Fig. \ref{Fig: Computation time of each iteration} provides the evolution of Algorithm \ref{Alg: Cutting generation algorithm} over iterations for both MINLP-based and MISOCP-based schemes.}

As can be seen, although the computation time is considerably diminished in the MISOCP-based model, the solution quality (in terms of load curtailment) is the same, and this is highly effective in precarious situations such as the emergency management of distribution networks, since prompt measures can keep electromechanical dynamics away from becoming stability threatening.
%Additionally, the amount of load curtailment is quite similar in the two schemes.
%
\vspace{-5pt}
\begin{table} [htb]
\centering
\caption{Comparison Between the MISOCP and MINLP Models}
\label{tab: Computation Time}
\begin{tabular}{|c|c|c|c|c|}
\hline
\multirow{2}{*}{\textbf{\begin{tabular}[c]{@{}c@{}}Islanding\\ Severity\\  {[}kW{]}\end{tabular}}} & \multicolumn{2}{c|}{\textbf{MISOCP-Based Scheme}}                                                                                         & \multicolumn{2}{c|}{\textbf{MINLP-Based Scheme}}                                                                                          \\ \cline{2-5} 
                                                                                                   & \begin{tabular}[c]{@{}c@{}}Curtailment\\  {[}kW{]}\end{tabular} & \begin{tabular}[c]{@{}c@{}}Computation \\ time {[}sec.{]}\end{tabular} & \begin{tabular}[c]{@{}c@{}}Curtailment\\  {[}kW{]}\end{tabular} & \begin{tabular}[c]{@{}c@{}}Computation\\  time {[}sec.{]}\end{tabular} \\ \hline
$2700 $           & $2248.4$                                                          & $57.21$                                                                   & $2248.8 $                           & $2978.8$                                                                   \\ \hline
$3200  $          & $2725 $                                                           &$ 52.74   $                                                                & $2725$                             & $7185.7$                                                                  \\ \hline
$3700$                                                                                               & $3208.4$                                                          &$ 73.58 $                                                                  & $3209.4$                                                          & $9593.5 $                                                                  \\ \hline
\end{tabular}
\end{table}
\subsubsection{\textcolor{myblue}{Convergence}}
In order to see more details about the convergence of Algorithm \ref{Alg: Cutting generation algorithm}, let us analyze the second islanding event (with the severity of $3200$ kW). For this event, Table \ref{tab: Convergence Process MISOCP} provides the objective function value, the cardinality of the set $\mathcal{F}_k$, the amount of load shedding, the configuration of the \MM ~network, \textcolor{myblue}{and the elapsed time} in each iteration of the algorithm while solving $\mathcal{MISOCP}(\mathcal{R},\mathcal{F})$. 
Accordingly, the algorithm converges in $15$ iterations. \textcolor{myblue}{In each iteration, a set of cuts are generated to separate a given solution of $\mathcal{MISOCP}(\mathcal{R},\mathcal{F}_k)$, that is a mixed integer solution, from the set $\mathcal{A}$. This separation in each iteration leads to an interplay between load shedding adjustments and network topology control, demonstrated in the $4^{th}$ and $5^{th}$ columns of Table \ref{tab: Convergence Process MISOCP}. It must be emphasized that when a mixed integer solution is cut off, the corresponding integer solution (i.e., the projection onto the space of integer variables) may not be cut off. For instance, in the $7^{th}$ iteration in in Table \ref{tab: Convergence Process MISOCP}, the amount of load shedding is $2295$~kW and the connected edges of the linking grid are $l_1$, $l_3$, and $l_5$ (see Fig. \ref{Fig: Multi Microgrid}(b)). Although a valid inequality cuts off this mixed integer solution in the next iteration, the corresponding integer solution appears again in the $15^{th}$ iteration with a different amount of load shedding.} 
% Note that each cut separates a given solution of $\mathcal{MISOCP}(\mathcal{R},\mathcal{F}_k)$, which is indeed a mixed-integer solution, from the set $\mathcal{A}$. This mixed integer solution includes both    
% Particularly,  demonstrate the 

As another interesting result, in the eighth iteration, the distribution network is partitioned into two sub-systems and the objective function is increased by $8.5$\%. Eventually, in the $15^{th}$ iteration, the optimal resilience improvement strategy is achieved while the distribution system is reconfigured as one connected component. 

For the sake of comparison, Table \ref{tab: Convergence Process in MINLP} provides the outputs of Algorithm \ref{Alg: Cutting generation algorithm} while solving $\mathcal{MINLP}(\mathcal{X},\mathcal{F})$. As can be seen, the algorithm converges in a more number of iterations and the computation time of each iteration is considerably more than that of the MISOCP-based model. The final solutions (the objective function, load curtailment, and configuration of the linking grid), nevertheless, are quite the same as the ones in Table \ref{tab: Convergence Process MISOCP}.       
   
%%%%%%%%
\begin{table} [htb]
\centering
\caption{Convergence Process of the Proposed Algorithm While Solving $\mathcal{MISOCP}(\mathcal{R},\mathcal{F})$}
\label{tab: Convergence Process MISOCP}
%\begin{tabular}{|c|c|c|c|c|}
\begin{tabular}{|M{0.19cm}|M{1cm}|M{0.39cm}|M{1.2cm}|M{1.9cm}|M{1.5cm}|}
\hline
$k$   &  \begin{tabular}[c]{@{}l@{}} $\quad \: \: \psi $ \\ $[ \times 100 \: \$ ]$ \end{tabular} & $|\mathcal{F}_k|$ & \begin{tabular}[c]{@{}l@{}} Curtailment \\ $[ \times 100 \: \textrm{kW} ]$ \end{tabular} & \textcolor{myblue}{Connected edges of $\NNN$} & \textcolor{myblue}{Elapsed time/iter [sec.]}
\\ \hline \hline
$0$        & $1579.09$         & $0$   & $22.98$         & $l_1,l_2,l_3 $	& $6.0$		 \\ \hline
$1$        & $1579.09$         & $1$   & $22.93$         & $l_2,l_4,l_5 $	& $3.8$		\\ \hline
$2$        & $1579.09$         & $2$   & $22.91$         & $l_2,l_3,l_4 $	& $4.0$	    \\ \hline
$3$        & $1579.09$         & $3$   & $23.07$         & $l_1,l_3,l_4 $	& $4.6$	   \\ \hline
$4$        & $1579.09 $        & $4$   & $22.87$         & $l_3,l_4,l_5 $	& $3.2$	   \\ \hline
$5$        & $1579.09 $        & $5$   & $22.93$         & $l_1,l_2,l_5 $	& $2.9$	   \\ \hline
$6$        & $1579.09 $        & $6$   & $23.03$         & $l_1,l_2,l_4 $	& $2.6$	    \\ \hline
$7$        & $1579.09 $        & $7$   & $22.95$         & $l_1,l_3,l_5 $	& $	2.7$	\\ \hline
$8$        & $1714.69 $        & $8$   & $23.40$         & $l_2,l_3,l_5	$	& $1.6	$	 \\ \hline
$9$        & $2069.38 $        & $10$  & $24.29$         & $l_1,l_4      $   & $5.9	$	 \\ \hline
$10$       & $2081.35 $        & $12$  & $24.92$         & $l_1,l_5     $    & $3.9 $		 \\ \hline
$11 $      & $2081.35 $        & $13$  & $24.92$         & $l_4,l_5     $    & $2.0	$	\\ \hline
$12 $      & $2086.23 $        & $14$  & $24.11$         & $l_2,l_5     $   	& $	2.4$	 \\ \hline
$13 $      & $2086.23 $        & $15$  & $24.15$         & $l_3,l_5     $    & $2.2	$	 \\ \hline
$14 $      & $2086.23 $        & $16$  & $24.07$         & $l_2,l_3     $    & $	2.7$	 \\ \hline
$15 $      & $2235.96 $        & $17$  & $27.25$         & $l_1,l_3,l_5 $	& $2.4	$	\\ \hline
\end{tabular}
\end{table}
%%%%%%%%
\vspace{-3pt}
\begin{table} [htb]
\centering
\caption{Convergence Process of the Proposed Algorithm While Solving $\mathcal{MINLP}(\mathcal{X},\mathcal{F})$}
\label{tab: Convergence Process in MINLP}
% \begin{tabular}{|c|c|c|c|c|}
% \hline
% $k$      &  \begin{tabular}[c]{@{}l@{}} $\psi $ \\ $[ \times 100 \: \$ ]$ \end{tabular} & $|\mathcal{F}|$ & \begin{tabular}[c]{@{}l@{}} Curtailment \\ $[ \times 100 \: \textrm{kW} ]$ \end{tabular} & \textcolor{myblue}{Edges of the linking grid}                                                                     \\ \hline
\begin{tabular}{|M{0.19cm}|M{1cm}|M{0.39cm}|M{1.2cm}|M{1.9cm}|M{1.5cm}|}
\hline
$k$   &  \begin{tabular}[c]{@{}l@{}} $ \quad \:\: \vartheta $ \\ $[ \times 100 \: \$ ]$ \end{tabular} & $|\mathcal{F}_k|$ & \begin{tabular}[c]{@{}l@{}} Curtailment \\ $[ \times 100 \: \textrm{kW} ]$ \end{tabular} & \textcolor{myblue}{Connected edges of $\NNN$} & \textcolor{myblue}{Elapsed time/iter [sec.]}
\\ \hline \hline
$0 $       & $1579.09  $       &$ 0  $ & $22.89  $       &$ l_1,l_3,l_4   $           &$664.2$    \\ \hline
$1 $       & $1579.09  $       &$ 1  $ & $22.88  $       &$ l_1,l_2,l_4    $          &$478.1$         \\ \hline
$2 $       & $1579.09  $       &$ 2  $ & $22.87  $       & $l_1,l_2,l_3,l_5$          &$391.9$          \\ \hline
$3 $       & $1579.09  $       &$ 3  $ & $22.86  $       & $l_1,l_2,l_5 $             &$329.8$            \\ \hline
$4 $       & $1579.09  $       &$ 4  $ & $22.87  $       & $ l_2,l_3,l_4,l_5$         &$341.4$           \\ \hline
$5 $       & $1579.09  $       &$ 5  $ & $22.84  $       &$ l_2,l_4,l_5  $            &$246.3$           \\ \hline
$6 $       & $1579.09  $       &$ 6  $ & $22.86  $       &$ l_1,l_3,l_5  $            &$285.5$         \\ \hline
$7 $       & $1579.09  $       &$ 7  $ & $22.83  $       &$ l_3,l_4,l_5  $            &$185.0$          \\ \hline
$8 $       & $1579.09  $       &$ 8  $ & $22.86  $       &$ l_1,l_2,l_3  $            &$165.0$          \\ \hline
$9 $       & $1579.09  $       &$ 9  $ & $22.86  $       &$ l_2,l_3,l_4  $            &$124.1$          \\ \hline
$10 $      & $1714.69  $       &$ 10 $ & $23.14  $       &$ l_2,l_3,l_5  $            &$291.7$          \\ \hline
$11 $      & $2069.38  $       &$ 12 $ & $24.06  $       &$ l_1,l_4        $          &$516.6$           \\ \hline
$12 $      & $2081.35  $       &$ 14 $ &$ 24.68  $       &$ l_1,l_5        $          &$387.5$           \\ \hline
$13 $      & $2081.35  $       &$ 15 $ & $ 24.71 $       &$ l_4,l_5       $           &$353.4$           \\ \hline
$14 $      & $2086.23  $       &$ 16 $ &$ 23.86  $       & $l_2,l_5        $          &$322.2$          \\ \hline
$15 $      & $2086.23  $       &$ 17 $ &$ 23.88  $       & $l_3,l_5        $          &$530.0$           \\ \hline
$16 $      & $2086.23  $       &$ 18 $ &$ 23.89  $       & $l_2,l_3        $          &$172.6$           \\ \hline
$17 $      & $2235.96  $       &$ 19 $ &$ 27.25  $       & $l_1,l_3,l_4    $          &$1400.7$           \\ \hline
\end{tabular}
\end{table}
\vspace{-7pt}
%%%%%%%%%%%%%%%
%%%%%%%%%%%%%%%%%%
\subsection{Comparison with the Conventional UFLS Scheme}

Fig. \ref{Fig: Comparison with conventional UFLS} provides a comparison between the MISOCP-based scheme and the conventional UFLS scheme while they are coping with the second islanding event \textcolor{myblue}{(with severity of $3200$~kW).} To have a more realistic result, we assume the communication latency to be $100$ ms in the proposed scheme. We also consider  the intentional delay of the UFLS relays to be $100$ ms. Since the distribution network is partitioned into four \M s in the conventional UFLS scheme, this figure compares the amount of load shedding, nadir frequency, and steady state frequency in each \M ~(\textcolor{myblue}{denoted by $m_1$ to $m_4$}), on the one hand, and the same indices in the multi-\M ~network which is obtained from the proposed MISOCP-based scheme, on the other hand.

Accordingly, the total amount of load shedding in our proposed scheme is $2725$ kW, while the steady state and nadir frequencies are remained within the permissible range. In comparison, the total amount of load shedding in the conventional scheme is $3700$ kW (even more than the initial power deficiency), and the frequency of the \M s violates the safe range. Specifically, in $m_3$, the violation of frequency is more serious, and the conventional scheme fails to maintain the frequency stability of the network. The main reason for this observation is the rigidity of the conventional UFLS scheme in dealing with different contingencies. In this scheme, load shedding is implemented in several steps with fixed sizes, regardless of the intensity of the islanding. Therefore, it can be inferred that the conventional method sheds non-optimal amount of loads encountering islanding events. These results illustrate that the proposed method is capable of preserving the distribution network from collapsing and moving it to a new steady state and stable condition. It is worth mentioning that, aside from the \textcolor{myblue}{COI frequency}, keeping the bus voltages and line flows within the permissible range in our proposed scheme would guarantee a secure operation following the islanding process, which is not considered in the conventional scheme.

\begin{figure} [t]
\includegraphics*[width=3.35 in, keepaspectratio=true]{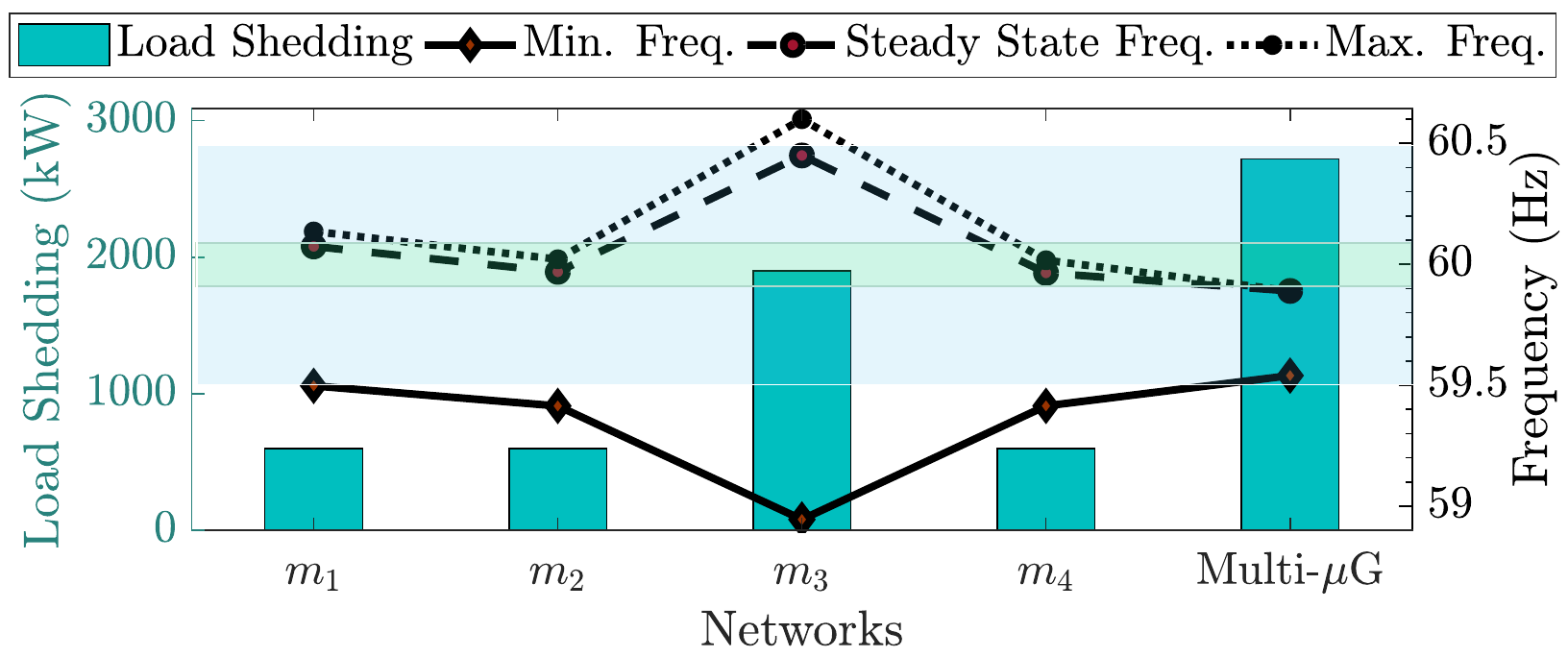}
    \centering
   \caption{Comparison between the proposed MISOCP-based and conventional UFLS schemes \textcolor{myblue}{for an islanding event with severity of $3200$~kW. Permissible ranges of nadir and steady state frequencies are shown by horizontal bars.}}
    \label{Fig: Comparison with conventional UFLS}
\end{figure}

%\vspace{-7pt}
\section{Conclusions} \label{Sec: Conclusions}
In this paper, we propose a novel framework for the near real-time operation as well as the real-time control of multi-\M ~networks. Our framework provides the optimal power flow, optimal load shedding, and optimal topology reconfiguration, while frequency dynamics and AC power flow limitations are taken into account. An exact reformulation of frequency constraints in a cutting plane algorithm with tight MISOCP relaxations is established, which significantly speeds up computation and achieves near optimal solution. To the best of our knowledge, this comprehensive optimization and control framework for the frequency stability of multi-\M s is proposed for the first time in the literature. Our numerical experiments further illustrate that the proposed emergency control scheme can successfully monitor, verify, and act to guarantee that the multi-\M ~network remains within the operational limits during post-islanding frequency dynamics. It is practical for real-world applications and outperforms the conventional UFLS scheme in terms of load shedding amount, number of curtailed customers, and frequency stability.

\bibliographystyle{IEEEtran}
\bibliography{References}

% \textbf{Amin Gholami }(S’14) is currently working towards the Ph.D. degree in the H. Milton Stewart School of Industrial and Systems Engineering, Georgia Tech, Atlanta, GA, USA. His research interests include power system optimization, stability, and control.

\end{document}